\documentclass[a4paper]{article}

\usepackage{amsmath,amssymb,amsthm,centernot,indentfirst,mathrsfs}
\usepackage{color}

\theoremstyle{plain} 
	\newtheorem{Thm}{Theorem}[section]
	\newtheorem{Cor}[Thm]{Corollary}
	\newtheorem{Lem}[Thm]{Lemma}
	
	\newtheorem{Prop}[Thm]{Proposition}
	
\theoremstyle{definition}
        \newtheorem{Def}[Thm]{Definition}
        
	\newtheorem{Rem}[Thm]{Remark}

\allowdisplaybreaks[4]
\numberwithin{equation}{section}

\newcommand{\R}{\mathbb{R}}
\newcommand{\m}{\mathfrak{m}}
\newcommand{\M}{\mathcal{M}}
\newcommand{\E}{\mathcal{E}}

\newcommand{\inn}[2]{\left\langle#1, #2\right\rangle}
\newcommand{\Ent}{\mathrm{Ent}}
\newcommand{\Ric}{\mathrm{Ric}}
\newcommand{\Hess}{\mathrm{Hess}}

\newcommand{\D}{\mathcal{D}}

\begin{document}

\title{Dilation type inequalities for strongly-convex sets in weighted Riemannian manifolds}
\author{Hiroshi Tsuji\footnote{Department of Mathematics, Osaka University, Osaka 560-0043, Japan (u302167i@ecs.osaka-u.ac.jp),  
	2020 {\it Mathematics Subject Classification: 53C20, 28A75, 26D10, 60E15,}  
	 }}
\date{}
\maketitle

\begin{abstract}
In this paper, we consider a dilation type inequality on a weighted Riemannian manifold, which is classically known as Borell's lemma in high-dimensional convex geometry. 
We investigate the dilation type inequality as an isoperimetric type inequality by introducing the dilation profile 
and estimate it by the one for the corresponding model space under lower weighted Ricci curvature bounds. 
We also explore functional inequalities derived from the comparison of the dilation profiles under the nonnegative weighted Ricci curvature.
In particular, we show several functional inequalities related to various entropies. 
\end{abstract}

\tableofcontents

\section{Introduction}

Let $\mu$ be an $s$-concave probability measure on $\R^n$ for $s \in [-\infty, \infty]$, which implies that 

\begin{align}\label{s-concavity}
\mu((1-t)A + tB) \geq \left( (1-t) \mu(A)^s + t\mu(B)^s \right)^{1/s}
\end{align}
holds for any compact subsets $A, B \subset \R^n$ with $\mu(A), \mu(B) >0$ and any $t \in [0, 1]$, where $(1-t)A + tB = \{ (1-t)a + tb ~|~ a \in A, b \in B\}$ is the Minkowski sum. 
The right hand side of \eqref{s-concavity} means $\min\{\mu(A), \mu(B)\}$ when $s=-\infty$, $\mu(A)^{1-t}\mu(B)^t$ when $s=0$ (in this case, $\mu$ is called log-concave), and $\max\{\mu(A), \mu(B)\}$ when $s=\infty$. 
Borell noticed in \cite{Bor-1} that, since we have  
\[
\frac{2}{t+1}(\R^n \setminus (tK)) + \frac{t-1}{t+1}K \subset \R^n \setminus K
\]
for any centrally symmetric convex subset $K \subset \R^n$ and $t \geq 1$, it follows from \eqref{s-concavity} that  
\begin{align}\label{dilation}
\left( \frac{2}{t+1} \mu(\R^n \setminus (tK))^s + \frac{t-1}{t+1}\mu(K)^s \right)^{1/s} \leq \mu(\R^n \setminus K)
\end{align}
when $\mu(K) > 0$ and $\mu(tK) < 1$. 
This inequality is called the dilation inequality or mentioned as Borell's lemma, and applied to high-dimensional convex geometry (for instance, see \cite{Bor-1}, \cite{L}, \cite{G}, \cite{BGVV}). 
However, the inequality \eqref{dilation} is not optimal for a convex subset $K$ with $\mu(K) \leq 1/2$.
Indeed, for instance, when $\mu$ is log-concave, the inequality \eqref{dilation} is equivalent to the form
\begin{align}\label{Borell's lemma}
1-\mu(tK) \leq \left( \frac{1-\mu(K)}{\mu(K)} \right)^{(t+1)/2} \mu(K),
\end{align}
and the right hand side above goes to $0$ as $t\to \infty$ if and only if $\mu(K) > 1/2$.
Lov\'{a}sz and Simonovits gave an optimal dilation inequality for log-concave probability measures and centrally symmetric subsets \cite[Theorem 2.8]{LS}, and later Gu\'{e}don \cite{G} proved by the localization method that  
\[
\left( \frac{2}{t+1} \mu(\R^n \setminus (tK))^s + \frac{t-1}{t+1} \right)^{1/s} \leq \mu(\R^n \setminus K)
\] 
for any $s$-concave probability measure $\mu$ with $0\leq s \leq 1/n$, centrally symmetric convex subset $K \subset \R^n$ and $t \geq 1$ (with $\mu(tK) <1$ when $s>0$). 

Moreover, the above dilation inequality was generalized for any Borel subset in $\R^n$ by Nazarov, Sodin and Vol'berg \cite{NSV}, Bobkov \cite{B1}, \cite{B2}, Bobkov and Nazarov \cite{BN}, and Fradelizi \cite{F} as follows. 
Given a Borel subset $A \subset \R^n$ and $t\geq 1$, we define $A^t \subset \R^n$ as 
\begin{align}\label{aa}
A^t := A \cup \left\{ x \in \R^n ~\Bigg{|}~ \text{there exists some interval } I \subset \R^n \text{ such that } x \in I \text{ and } |I \cap A| > \frac{2}{t+1} |I| \right\},
\end{align}
where $|\cdot|$ means the $1$-dimensional Lebesgue measure. 
We may assume that  $x$ is an endpoint of $I$ in the definition of \eqref{aa}.
Note that $A^t$ is a Borel set and $A^1=A$. In addition, 
when $A$ is an open convex subset in $\R^n$, Fradelizi \cite[Fact 1]{F} showed that for any $t \geq 1$, 
\[
A^t =A + \frac{t-1}{2}(A-A).
\]
Therefore, when $A$ is symmetric centered at $a \in \R^n$, then $A^t = t(A-a) + a$. 
In particular, when $a=0$, $A^t$ coincides with $tA$, and hence we may consider the set defined by \eqref{aa} as a generalization 
of the dilation for centrally symmetric convex subsets. 
For other detailed properties of the dilation defined as \eqref{aa}, see \cite{F}. 
The following inequality is the dilation inequality on the dilation set defined by \eqref{aa}: 
given an $s$-concave probability measure $\mu$ on $\R^n$ with $s \leq 1/n$,  it holds that 
\begin{align}\label{a}
\left( \frac{2}{t+1}\mu(\R^n \setminus A^t)^s + \frac{t-1}{t+1} \right)^{1/s} \leq \mu(\R^n \setminus A)
\end{align}
for any Borel subset $A \subset \R^n$ and $t\geq 1$ (with $\mu(A^t) <1$ when $s>0$).
Note that the inequality \eqref{a} is sharp. Indeed, when $\mu_s$ is the probability measure on $\R$ whose density with respect to the 1-dimensional Lebesgue measure is 
\begin{align}\label{aaaaa}
\rho_s(x) := (1-sx)_+^{(1-s)/s} \mathbf{1}_{[0, \infty)}(x),
\end{align}
where $(\cdot)_+ := \max\{ \cdot, 0\}$, then $\mu_s$ is $s$-concave and equality holds in \eqref{a} for any interval $[0, b] \subset \R$ with $b >0$. 

We comment on the methods of the preceding studies. 
Bobkov \cite{B1}, \cite{B2} showed a weak type of \eqref{a} using triangular maps in mass transport theory, and Bobkov and Nazarov \cite{BN} and Fradelizi \cite{F} used the localization method to prove \eqref{a}. 
Recently, this localization method was extended to weighted Riemannian manifolds by Klartag \cite{K} through optimal transport theory (more precisely, this extension corresponds to the localization by Lov\'{a}sz and Simonovits \cite{LS} which is used in \cite{BN}, 
however, Fradelizi used the ``geometric'' localization method (see \cite{FG1}, \cite{FG2} for more information) in which we need to use the Krein-Milman theorem). 
Since the characterization of densities of $s$-concave probability measures on $\R^n$ by Borell \cite{Bor-1} implies that the $s$-concavity of measures is characterized by non-negativity of the weighted Ricci curvature, 
the inequality \eqref{a} is also established on weighted Riemannian manifolds with nonnegative weighted Ricci curvature as Klartag mentioned in \cite[p.65]{K} as follows. 

Let $(\M, g)$ be an $n$-dimensional Riemannian manifold.
For any Borel subset $A \subset \M$ and $\varepsilon \in [0, 1)$, we define the {\it $\varepsilon$-dilation set} $A_{\varepsilon}$ of $A$ on $\M$ by 
\begin{align}\label{dilation2}
A_{\varepsilon} := A \cup \left\{ x \in \M | \text{there exists a minimizing geodesic } \gamma: [0, 1] \to \M \text{ with } \gamma(0) = x \text{ and } |\mathrm{\gamma} \cap A| > 1-\varepsilon \right\}, 
\end{align}
where $|\mathrm{\gamma} \cap A|$ means the 1-dimensional Lebesgue measure of the set $\{ t \in [0, 1]~|~ \gamma(t) \in A\}$. 
When $(\M, g) =(\R^n, \|\cdot\|_2^2)$ where $\|\cdot\|_2^2$ is the standard Euclidean norm, letting $t:=(1+\varepsilon)/(1-\varepsilon)$, we see that $A^t=A_{\varepsilon}$ for any Borel subset $A \subset \R^n$. 

\begin{Thm}[Klartag \cite{K}]\label{Klartag}
Let $(\M, g, \m)$ be a geodesically-convex $n$-dimensional Riemannian manifold with a weighted measure $\m$ satisfying $\m(\M)=1$.
If $(\M, g, \m)$ satisfies $\Ric_N \geq 0$ for some $N \in (-\infty, 0)\cup[n ,\infty]$, then it holds that for any $\varepsilon \in [0 ,1)$, whenever $\m(A_{\varepsilon})<1$, 
\begin{align}\label{aaa}
1- \m(A) \geq \left\{ (1-\varepsilon) \m(\M \setminus A_{\varepsilon})^{1/N} + \varepsilon \right\}^N.
\end{align}
When $N=\infty$, the right hand side  of \eqref{aaa} is interpreted as $\m(\M \setminus A_{\varepsilon})^{1-\varepsilon}$.
\end{Thm}
$\Ric_N$ is the weighted Ricci curvature which is defined in section \ref{WRM}.
Note that Theorem \ref{Klartag} recovers the dilation inequality for $s$-concave probability measures in the Euclidean setting for $s \in (-\infty, 1/n]$. 
Indeed, in virtue of the characterization of the $s$-concavity by Borell \cite{Bor}, we see that every $s$-concave probability measure on $\R^n$ for $s\leq 1/n$ satisfies $\Ric_{1/s}\geq 0$ on its support (when $s=0$, we put $1/s:=\infty$).
More generally, the lower curvature bound $\Ric_N \geq K$ on weighted Riemannian manifolds is known to be equivalent to the {\it curvature-dimension condition} in the sense of Lott-Sturm-Villani (see \cite{vRS}, \cite{St}, \cite{Oh1}, \cite{Oh2}). 

The main purpose of this paper is to establish the sharp dilation type inequalities under more general curvature conditions, namely $\Ric_N \geq K$ for some $K \in \R$. 
In our setting, we consider the dilation inequality \eqref{aaa} as an isoperimetric type inequality. Now, we introduce the dilation profile. 
For every $\varepsilon \in [0, 1)$ and $\theta \in [0, 1]$, we define the {\it $\varepsilon$-dilation profile of $(\M, g, \m)$} by 
\[
\mathcal{D}_{(\M, g, \m)}^{\varepsilon}(\theta) := \inf \{ \m(A_{\varepsilon}) ~|~ \text{a Borel subset } A \subset \M \text{ with } \m(A) = \theta \}.
\]
For instance, considering $(\R, |\cdot|, \mu_s)$ with $s \in (-\infty, 1/n]$ where $\mu_s$ is the $s$-concave probability measure defined by \eqref{aaaaa}, 
since it is the extremal of \eqref{a}, we see that
\begin{align}\label{aaaaaa}
\mathcal{D}_{(\R, |\cdot|, \mu_s)}^{\varepsilon}(\theta) = 1- \left( \frac{(1-\theta)^s - \varepsilon}{1-\varepsilon} \right)_+^{1/s},
\end{align}
where we set $0^{\alpha}:=1$ for $\alpha > 0$ by convention.
When $s =0$, the right hand side of \eqref{aaaaaa} is interpreted as $1- (1- \theta)^{1/ (1- \varepsilon)}$.
Note that \eqref{aaa} in Theorem \ref{Klartag} may be represented as $\D_{(\M, g, \m)}^{\varepsilon} \geq \D_{(\R, |\cdot|, \mu_{1/N})}^{\varepsilon}$ on $[0, 1]$ for any $\varepsilon \in [0, 1)$. 

In this paper, in addition to the dilation profile associated with $\varepsilon \in [0, 1)$, we also treat the following dilation profile: for any Borel subset $A \subset \M$, we define the {\it dilation area of $A$} by 
\[
\m^*(A):= \liminf_{\varepsilon \to 0} \frac{\m(A_{\varepsilon}) - \m(A)}{\varepsilon},
\]
and the {\it dilation profile of $(\M, g, \m)$} by 
\[
\D_{(\M, g, \m)}(\theta):= \inf \{ \m^*(A) ~|~ \text{a Borel subset } A\subset \M \text { with } \m(\M)=\theta\}
\]
for any $\theta \in [0,1]$. By this definition, under the same assumptions as in Theorem \ref{Klartag}, \eqref{aaaaaa} implies that 
\begin{align}\label{b}
\D_{(\M, g, \m)}(\theta) \geq \D_{(\R, |\cdot|, \mu_{1/N})}(\theta) = -N\left(1-\theta-(1-\theta)^{1-1/N}\right)
\end{align}
holds for any $\theta \in[0,1]$. Here, when $N=\infty$, the right hand side above is interpreted as $-(1-\theta)\log(1-\theta)$.
Note that the dilation profile differs from the isoperimetric profile. 
In fact, the dilation profile is scale invariant, namely $\D_{(\M, \lambda^2 g, \m)} = \D_{(\M, g, \m)}$ for any $\lambda>0$ since the $\varepsilon$-dilation is also scale invariant, while the isoperimetric profile $\mathcal{I}_{(\M, g, \m)}$ (for instance, see \cite{Mi1}, \cite{Mi2} for the precise definition and overviews) satisfies 
$\mathcal{I}_{(\M, \lambda^2 g, \m)} = \mathcal{I}_{(\M, g, \m)}/\lambda$ for any $\lambda>0$.

In order to describe our results, we introduce the following notations: 
given a nonzero integrable function $f :\R \to \R_+$ and an interval $[a, b]$ with $-\infty<a<b\leq \infty$ and $\int_a^b f(t) ~dt>0$, we define the {\it flat dilation profile} by
\[
\D^{\flat}(f, [a, b])(\theta) := \frac{f(\alpha(\theta))}{\int_a^b f(t) ~dt}(\alpha(\theta)-a)
\] 
for $\theta \in [0, 1]$, where $\alpha(\theta) \in [a, b]$ is given by 
\[
\theta = \frac{\int_a^{\alpha(\theta)} f(t) ~dt}{\int_a^b f(t) ~dt}.
\]
We also denote 
\begin{align*}
\mathfrak{s}_{\kappa}(t) := 
	\begin{cases}
		\frac{1}{\sqrt{\kappa}} \sin(\sqrt{\kappa}t)  &\text{if } \kappa>0, \\
		t &\text{if } \kappa=0, \\
		\frac{1}{\sqrt{-\kappa}} \sinh(\sqrt{-\kappa}t) &\text{if } \kappa<0, 
	\end{cases}
\quad
\mathfrak{c}_{\kappa}(t) :=
	\begin{cases}
		\cos(\sqrt{\kappa}t) &\text{if } \kappa>0, \\
		1 &\text{if } \kappa=0, \\
		\cosh(\sqrt{-\kappa}t) &\text{if } \kappa<0
	\end{cases}
\end{align*}
for all $\kappa, t \in \R$, and  
\begin{align*}
J_{H, K, N}(t) := 
	\begin{cases}
		\left( \mathfrak{c}_{\delta}(t) + \frac{H}{N-1}\mathfrak{s}_{\delta}(t)  \right)_+^{N-1} &\text{if } N \notin \{1, \infty\}, \\
		\exp(Ht -\frac{K}{2}t^2) &\text{if } N=\infty, \\
		1 &\text{if } N=1, K=0, \\
		\infty &\text{otherwise}
	\end{cases}
\end{align*}
for all $H, K, t \in \R$ and $N \in (-\infty, \infty]$, where $\delta := K/(N-1)$.
Now, we also define the {\it Curvature-Dimension-Diameter (CDD) dilation profile} by  
\[
\mathscr{D}_{K, N, D}(\theta) := \inf_{(a, b) \in \Delta_D, H \in \R} \max \left\{ \frac{a(1-\theta)}{\int_0^{b} J_{H, K, N}(t) ~dt}, \frac{a\theta}{\int_{-a}^0 J_{H, K, N}(t) ~dt} \right\}
\]
for any $K,N \in \R$, $D \in (0, \infty]$ and $\theta \in [0, 1]$, where 
\begin{align*}
\Delta_D :=
	\begin{cases}
		\{ (a, b) ~|~ a, b >0, a+b =D\} &\text{if $D<\infty$},\\
		 \{ (a, \infty) ~|~ a>0 \} &\text{if $D=\infty$}.
	\end{cases}
\end{align*}

\begin{Thm}\label{MT0}
Let $(\M, g, \m)$ be a geodesically-convex $n$-dimensional weighted Riemannian manifold satisfying $\m(\M)=1$, $\Ric_N\geq K$ and $\mathrm{diam} \M \leq D$ for some $K\in \R$, $N \in (-\infty, 1)\cup [n, \infty]$ and $D \in (0, \infty]$ ($N\neq 1$ when $n=1$). 
Then for any strongly-convex subset $A \subset \M$, it holds that 
\[
\m^*(A) \geq \mathscr{D}_{K, N, D}(\m(A)).
\]
\end{Thm}

We say that $A\subset \M$ is {\it strongly-convex} if for any $p, q\in A$, there exists a unique minimizing geodesic connecting $p$ and $q$ and is included in $A$.
We also define the diameter of $\M$ by $\mathrm{diam}\M := \sup \{ d_g(x, y) ~|~x, y \in \M\}$, where $d_g$ is the distance function canonically induced by $g$.
In some special cases, the CDD dilation profiles have more concrete representations. 

\begin{Cor}\label{MT1}
Let $(\M, g, \m)$ be a geodesically-convex $n$-dimensional weighted Riemannian manifold satisfying $\m(\M)=1$, $\Ric_N\geq K$ and $\mathrm{diam} \M \leq D$ for some $K\in \R$, $N \in (-\infty, 0]\cup [n, \infty]$ and $D \in (0, \infty]$ ($N \neq 1$ when $n=1$).
Then there exists some function $\D_{K, N, D}$ on $[0, 1]$ depending only on $K, N$ and $D$ such that 
$\m^*(A) \geq \D_{K, N, D}(\m(A))$ holds for any strongly-convex subset $A \subset \M$, where the function $\D_{K, N, D}$ is given as follows: 

{\bf Case 1.} If $N=\infty$, $K> 0$ and $D=\infty$, 
\[
\D_{K, N, D} := \inf_{x \in \R} \D^{\flat}(e^{-Kt^2/2}, [x, \infty)).
\]

{\bf Case 2.} If $N=\infty$, $K\neq 0$ and $D<\infty$, 
\[
\D_{K, N, D} := \inf_{x \in \R} \D^{\flat}(e^{-Kt^2/2}, [x, x+D]).
\]

{\bf Case 3.} If $N=\infty$ and $K=0$, 
\[
\D_{K, N, D}(\theta) := \D^{\flat}(e^{-t}, [0, \infty))(\theta) = -(1-\theta)\log(1-\theta)
\]
for any $\theta \in [0,1]$.

{\bf Case 4.} If $N \in [n, \infty)$ and $K>0$, 
\[
\D_{K, N, D} := \inf_{x \in [0, \pi/\sqrt{\delta})} \D^{\flat} (\sin^{N-1}(\sqrt{\delta}t), [x, \min\{x+D, \pi/\sqrt{\delta}\}]).
\]

{\bf Case 5.} If $N \in [n ,\infty)$ and $K=0$, 
\[
\D_{K, N, D}(\theta) := \D^{\flat}( (-t)^{N-1}, [-1, 0])(\theta) = -N(1-\theta-(1-\theta)^{1-1/N})
\]
for any $\theta \in [0,1]$. 

{\bf Case 6.} If $N \in [n, \infty)$, $K<0$ and $D <\infty$,
\[
\D_{K, N, D} :=
\min \left\{
	\begin{array}{ll}
		\inf_{x \in (-\infty, 0)} \D^{\flat}( \sinh^{N-1}(-\sqrt{-\delta}t), [x, \min\{x+D, 0\}]), \\
		\inf_{x \in \R} \D^{\flat}( \cosh^{N-1}(\sqrt{-\delta}t), [x, x+D]), \\
		\D^{\flat}( e^{-\sqrt{-\delta}(N-1)t}, [0, D])
	\end{array}
\right\}.
\]

{\bf Case 7.} If $N \in (-\infty,0]$ and $K>0$,
\begin{align*}
\D_{K, N, D} := \min \left\{
	\begin{array}{ll}
		\inf_{x >0} \D^{\flat}( \sinh^{N-1}(\sqrt{-\delta}t), [x, x+D]), \\
		\inf_{x \in \R} \D^{\flat}( \cosh^{N-1}(\sqrt{-\delta}t), [x, x+D]), \\
		\D^{\flat}( e^{\sqrt{-\delta}(N-1)t}, [0, D])
	\end{array}
\right\},
\end{align*}
where we put $[x, x+D]:=[x, \infty)$ when $D=\infty$.

{\bf Case 8.} If $N \in (-\infty,0)$ and $K=0$,
\[
\D_{K, N, D}(\theta) := \D^{\flat}( t^{N-1}, [1, \infty))(\theta) = -N(1-\theta-(1-\theta)^{1-1/N})
\]
for any $\theta \in [0, 1]$.

{\bf Case 9.} If $N \in (-\infty,0]$, $K<0$ and $D < \pi/\sqrt{\delta}$,
\[
	\D_{K, N, D} := \inf_{x \in (0, \pi/\sqrt{\delta}-D)} \D^{\flat} (\sin^{N-1}(\sqrt{\delta}t), [x, x+D]).
\]
\end{Cor}
In the above corollary, the infimums are considered in the pointwise sense. 
We can also observe that $\mathscr{D}_{K, N, D}$ coincides with $\D_{K, N, D}$ for a triple $(K, N, D)$ in Corollary \ref{MT1}. 
Note also that the range of a triple $(K, N, D)$ discussed in Corollary \ref{MT1} is derived from the continuity in $H \in \R$ of 
entries in $\mathscr{D}_{K, N, D}$ (see the proof of Corollary \ref{MT1} and Remark \ref{rem: rem1}).
We note some remarks of Corollary \ref{MT1}.
In general, when $(\M, g, \m)$ satisfies the {\it Curvature-Dimension-Diameter (CDD) condition}, namely $\Ric_N \geq K$ and $\mathrm{diam}\M\leq D$, then for any $\lambda>0$, 
$(\M, \lambda^2 g, \m)$ satisfies $\Ric_N \geq K/\lambda^2$ and $\mathrm{diam}\M \leq \lambda D$. 
Thus, since the dilation area is invariant under scaling, we see that $\D_{K, N, \infty}$ in Cases 1 and 4 coincides with $\D_{1, N, \infty}$
and that $\D_{0, N, D}$ in Cases 3, 5 and 8 are independent of $D \in (0, \infty]$. 
In particular, we emphasize that $\D_{K, N, \infty}$ does not converge to $\D_{0, N, \infty}$ as $K\to 0$. 

This paper is organized as follows. 
In section 2, we introduce the weighted Ricci curvature and the localization on a weighted Riemannian manifold. 
In section 3, we discuss the dilation inequality on $\R$. In the first subsection, we give sufficient conditions such that the infimum of the dilation profile is attained at an interval. As its corollary, we obtain an explicit representation of the Gaussian dilation profile on $\R$.
In the next subsection, we complete the proofs of Theorem \ref{MT0} and Corollary \ref{MT1}. 
In section 4, we furthermore discuss the dilation inequality associated with $\varepsilon$. 
In the final section, we prove a new type of functional inequalities related to entropies, derived from the comparison of the dilation profiles \eqref{b}. 
 
\section*{{\bf Acknowledgments}}
The author would like to thank Professor Shin-ichi Ohta 
for helpful comments and supports. 
This work was supported by JST, ACT-X Grant Number JPMJAX200J, Japan.

\section{Preliminaries for weighted Riemannian manifolds}\label{WRM}

\subsection{Localization associated with lower weighted Ricci curvature bounds}
In this subsection, we introduce some notions on weighted Riemannian manifolds and the needle decomposition (also called the localization) constructed by Klartag in \cite {K}. Using this decomposition, we can reduce our problem to the $1$-dimensional one. 

Let $(\M, g, \m)$ be a geodesically-convex (namely, every two points can be connected by a minimizing geodesic) $n$-dimensional weighted  Riemannian manifold with $\m=e^{-\Psi} \mathrm{vol}_g$ and $\Psi \in C^{\infty}(\M)$, where $\mathrm{vol}_g$ is the canonical Riemannian volume on $\M$ induced by $g$. 
For $N \in (-\infty, \infty]$, the weighted Ricci curvature $\Ric_N$ is defined as

\begin{description}
\item{(1)} $\displaystyle \Ric_N(v) := \Ric_g(v) + \Hess \Psi(v, v) -\frac{\inn{\nabla \Psi(x)}{v}^2}{N-n}$ \quad \text{ if } $N\neq n, \infty$,  
\item{(2)} $\displaystyle \Ric_{\infty}(v) := \Ric_g(v) + \Hess \Psi(v, v)$,
\item{(3)} $\Ric_n(v) :=
\begin{cases}
\Ric_g(v) + \Hess \Psi(v, v)  &\quad \text{ if }  \inn{\nabla \Psi(x)}{v}=0, \\
-\infty &\quad \text{otherwise}
\end{cases}$
\end{description}
for any $p\in \M$ and $v \in T_p\M$, where $\Ric_g$ is the Ricci curvature on $\M$ canonically induced by $g$. 
We say that $(\M, g, \m)$ satisfies the {\it Curvature-Dimension (CD) condition $CD(K, N)$}, or $\Ric_N \geq K$, for some $K\in \R$ and $N \in (-\infty, \infty]$ if $\Ric_N (v)\geq K g(v, v)$ for any $v \in T\M$. 
Simple observations yield that 
\[
\Ric_{n}\leq \Ric_N\leq \Ric_{\infty}\leq \Ric_{N'}
\]
for any $N \in (n ,\infty)$ and $N' \in (-\infty, 1)$. 

The following theorem is the needle decomposition proved by Klartag \cite{K} on a weighted Riemannian manifold, which has a lot of geometric and analytic applications (for instance \cite{Ma}, \cite{OT}, \cite{MO}), and its extensions and applications in more general spaces are also investigated (see \cite{CM}, \cite{Oh2}). 
\begin{Thm}[\cite{K}]\label{localization}
Let $n \geq 2$, $K \in \R$, $N \in (-\infty, 1) \cup [n, \infty]$ and $(\M, g, \m)$ be a geodesically-convex $n$-dimensional Riemannian manifold satisfying $CD(K, N)$. 
Assume that $f: \M\to \R$ is an integrable function with $\int_\M f ~d\m =0$, and that there exists some $x_0 \in \M$ satisfying $\int_\M |f(x)| d_g(x, x_0) ~d\m(x) <\infty $. Then, there exist a partition $Q$ of $\M$, a measure $\nu$ on $Q$ and a family $\{\mu_I\}_{I \in Q}$ of probability measures on $\M$ satisfying the following:
\begin{itemize}
\item[(i)] For any Lebesgue measurable set $A \subset \M$, 
\[
\m(A) = \int_Q \mu_I(A) ~d\nu(I).
\]
\item[(ii)] For $\nu$-almost every $I \in Q$, $I$ is a minimizing geodesic in $\M$, and $\mu_I$ is supported on $I$. Moreover if $I$ is not a singleton, then the density of $\mu_I$ is smooth, and $(I, |\cdot|, \mu_I)$ satisfies $CD(K, N)$. 
\item[(iii)] For $\nu$-almost every $I \in Q$, $\int_I f ~d\mu_I =0$ holds.
\end{itemize} 
\end{Thm}

In virtue of this localization, we can reduce our main assertion to the 1-dimensional one. Thus, we will discuss the dilation inequality on $\R$ in the next section. 

Now, we also recall that the $\varepsilon$-dilation $A_{\varepsilon}$ of a Borel subset $A \subset \M$ is defined by \eqref{dilation2}. 
Note that $A_\varepsilon$ is also a Borel subset. 
Indeed, setting 
\[
c(\xi) := \sup\{ t \geq 0~|~ \text{$\exp_p(sv)$ is a minimizing geodesic for $s\in [0, t]$ in $\M$}\}>0
\]
for every $\xi = (p, v) \in T\M$ and $X:=\{ \xi \in T\M ~|~ g(\xi, \xi) < c(\xi)\}$, we have, for any minimizing geodesic $\gamma_\xi(s) =\exp_p(sv)$ with $s \in [0, 1]$ and $\xi =(p, v) \in X$, 
\[
\phi_A(\xi):=|\gamma_\xi \cap A| = \int_0^1 \mathbf{1}_A(\gamma_{\xi}(s)) ~ds. 
\]
Since $\phi_A$ is Borel measurable on $X$ and $c$ is continuous, the set 
$X_{\varepsilon}:=\{ \xi \in X~|~ \phi_A(\xi)>1-\varepsilon\}$ is a Borel set, and hence $A_{\varepsilon} = A \cup \pi(X_{\varepsilon})$
is also a Borel set, where $\pi : T\M \to \M$ is the canonical projection. 

\subsection{$(K, N)$-convex functions}

Let $(\M, g)$ be a geodesically-convex $n$-dimensional Riemannian manifold.
For $K \in \R$ and $N \in \R \setminus \{0\}$, we say that a function $\psi \in C^2(\M)$ is {\it $(K, N)$-convex} if 
\[
\mathrm{Hess} \psi(v, v) - \frac{\inn{\nabla \psi}{v}^2}{N} \geq K|v|^2
\]
for any $v \in T\M$.
According to \cite{EKS} for $N>0$ and \cite{Oh1} for $N<0$, there exists an equivalent representation as follows. 

\begin{Lem}[{\cite[Lemma 2.2]{EKS}, \cite[Lemma 2.1]{Oh1}}]\label{le:convex}
Let $(\M, g)$ be a geodesically-convex $n$-dimensional Riemannian manifold, $K \in \R$ and $N \in \R \setminus \{0\}$. For any $\psi \in C^2(\M)$, the following are equivalent: 
	\begin{itemize}
		\item[(i)] $\psi$ is $(K,N)$-convex. 
		\item[(ii)] For any non-constant minimizing geodesic $\gamma : [0, 1] \to \M$, with $d :=d_g(\gamma(0), \gamma(1)) < \pi\sqrt{N/K}$ when $N/K >0$, we have
			\[
				\psi_N(\gamma(1))^N \leq \left(\mathfrak{c}_{K/N}(d)\psi_N(\gamma(0)) + \frac{\mathfrak{s}_{K/N}(d)}{d}(\psi_N \circ \gamma)'(0) \right)_+^N, 
			\]
		where $\psi_N \in C^2(\M)$ is given by 
			\[
				\psi_N(x) := e^{-\psi(x)/N}.
			\]
	\end{itemize}
\end{Lem}

In particular, when a (probability) measure $\mu$ supported on an open interval $I\subset \R$ with a smooth density $e^{-\psi(x)}$ satisfies $\Ric_N \geq K$ for some $K \in \R$ and $N \in (-\infty, 1)\cup (1, \infty)$, then $\psi$ is $(K, N-1)$-convex. Thus, for any $x \in I$ and $t \in \R$ with $x+t \in I$, Lemma \ref{le:convex} implies 
\[
e^{-\psi(x+t)} \leq e^{-\psi(x)} J_{-\psi'(x), K, N}(t).
\]


\section{Estimates for dilation areas}

In this section, we discuss the dilation profile of a 1-dimensional weighted Riemannian manifold $(I, |\cdot|, \mu)$ with $\mu(I)=1$, where $I$ is an open interval in $\R$. 
For simplicity, we denote $\D_{(I, |\cdot|, \mu)}$ by $\D_{\mu}$. 

Before discussing the dilation inequality, we remark on the $\varepsilon$-dilation on $\R$. In general, given proper subsets $A \subset I \subset \R$ and $\varepsilon \in (0, 1)$, 
the $\varepsilon$-dilation of $A$ in $I$, denoted by $A_{\varepsilon}^1$, does not coincide with the one in $\R$, denoted by $A_{\varepsilon}^2$, since the former is necessarily included in $I$. 
However,  we can observe that $A_{\varepsilon}^1 = A_{\varepsilon}^2 \cap I$, and since we consider only a form $\mu(A_{\varepsilon})=\mu(A_{\varepsilon}\cap I)$ in our discussions, where $\mu$ is a (probability) measure supported on $I$, we consider the $\varepsilon$-dilation only in $\R$ even if the support of a discussed measure is not the whole space. 
The same problem occurs in more general spaces.

\subsection{Existence and properties of minimizer on the real line}

In this subsection, we consider sufficient conditions for a probability measure on $\R$ whose minimizer attaining the infimum of the dilation profile is an interval. In particular, our conditions will be satisfied by the Gaussian measures. 
\begin{Prop}\label{char_of_min}
Let $\mu$ be a probability measure on $\R$ whose density is $e^{-\psi}$ with $\psi \in C^1(\R)$. 
Assume that there exists some $\xi \in \R$ such that $\psi$ is non-increasing on $(-\infty ,\xi]$ and non-decreasing on $[\xi, \infty)$. 
In addition, we assume that for any $x, y \in \R$ with $x < y$, $\psi(x) \leq \psi(y)$ yields that
\begin{align}\label{sinh+}
\frac{\sinh(\psi(y)-\psi(x))}{y-x} \geq \frac{\psi'(y) +\psi'(x)}{2},
\end{align}
and $\psi(x) \geq \psi(y)$ yields that 
\begin{align}\label{sinh-}
\frac{\sinh(\psi(y)-\psi(x))}{y-x} \leq \frac{\psi'(y) +\psi'(x)}{2}. 
\end{align}
Given $\theta \in [0, 1]$, let $A_{\theta} \subset \R$ be an interval with $\mu(A_{\theta})=\theta$ whose endpoints $a_{\theta}, b_{\theta} \in \R$ satisfy $\psi(a_{\theta})=\psi(b_{\theta})$. 
Then for any $\theta \in [0, 1]$ and interval $A \subset \R$ with $\mu(A)=\theta$, we have $\mu^*(A) \geq \mu^*(A_{\theta})$.

\end{Prop}

\begin{proof}
Since the assertion is clear when $\theta=0$ and $1$, we may assume that $\theta \in (0, 1)$.
For fixed $\theta \in (0, 1)$, 
we will prove that for any interval $A$ with $\mu(A)=\theta$, $\mu^*(A) \geq \mu^*(A_\theta)$ holds.
Without loss of generality, we may assume that $A$ is open. Let $A=(a, b)$ and take $\varepsilon \in (0,1)$.
By the definition of the dilation,  we obtain 
\begin{align}\label{interval0}
A_{\varepsilon} = \left(a-\frac{\varepsilon}{1-\varepsilon}(b-a), b+\frac{\varepsilon}{1-\varepsilon}(b-a) \right),
\end{align}
and hence
\begin{align}\label{interval}
\mu^*(A) = \lim_{\varepsilon \to 0}\frac{1}{\varepsilon}\left( \mu\left( \left[a-\frac{\varepsilon}{1-\varepsilon}(b-a), a \right]\right) +\mu\left( \left[b, b+\frac{\varepsilon}{1-\varepsilon}(b-a) \right]\right) \right) = (e^{-\psi(a)}+e^{-\psi(b)})(b-a). 
\end{align}
Now, we consider a function $g$ on $\R$ satisfying $\mu([a, b]) = \mu([a + s, b + g(s)])$ for $s \in \R$. 
Then we have $g'(s)e^{-\psi(b + g(s))}=e^{-\psi(a+s)}$. 
Thus, we obtain 
\begin{align}
 &\frac{d}{ds}\mu^*([a + s, b + g(s)]) \notag \\
&=\frac{d}{ds} \left[ \left(e^{-\psi(a+s)} +e^{-\psi(b+g(s))}\right)((b+g(s))-(a+s)) \right] \notag \\
 &=\left(e^{-\psi(a+s)}+e^{-\psi(b+g(s))}\right)(g'(s)-1)  \notag \\
 &\hspace{2cm}  -\left(\psi'(a+s)e^{-\psi(a+s)} + \psi'(b + g(s))g'(s)e^{-\psi(b + g(s))}\right) (b+g(s)-a-s) \notag \\
 &= 2e^{-\psi(a+s)}(b+g(s)-a-s)\left\{ \frac{\sinh(\psi(b+g(s))-\psi(a+s))}{(b+g(s))-(a+s)}  - \frac{\psi'(b + g(s)) +\psi'(a+s)}{2} \right\} \label{bb}.
\end{align}
When $\psi(b) \geq \psi(a)$, it follows from \eqref{sinh+} and the monotonicity of $\psi$ that \eqref{bb} is nonnegative for any $s \geq 0$, which implies that $\mu^*([a + s, b + g(s)])$ is non-decreasing in $s \geq 0$. 
Similarly, when $\psi(b) \leq \psi(a)$, it follows from \eqref{sinh-} and the monotonicity  of $ \psi$ that $\mu^*([a + s, b + g(s)])$ is non-increasing in $s \leq 0$. 
Therefore we obtain $\mu^*(A) \geq \mu^*(A_\theta)$. 
\end{proof}

Now, we give some examples satisfying the assumptions in Proposition \ref{char_of_min}. 
Note that \eqref{sinh-} follows from \eqref{sinh+} when $\psi$ is even. 

An important example is the Gaussian measures. Let $\psi(x) := Kx^2/2 + \log\sqrt{2\pi/K}$ for some $K>0$.
Then for any $x, y\in \R$ with $x<y$ and $x^2 \leq y^2$, 
\[
\frac{\sinh(\psi(y)-\psi(x))}{y-x} > \frac{K}{2} \cdot \frac{y^2-x^2}{y-x} = \frac{K}{2}(y+x) = \frac{\psi'(y)+\psi'(x)}{2}.
\]
Thus the Gaussian measures satisfy \eqref{sinh+}.

More generally, if $\psi$ is symmetric centered at $\xi \in \R$ in $C^2(\R)$ and non-decreasing on $[\xi, \infty)$ and $\psi''$ is non-increasing on $[\xi, \infty)$, then $\psi$ satisfies \eqref{sinh+}. 
Indeed, for $x, y\in \R$ with $x <y$ and $\psi(x) \leq \psi(y)$, we obtain 
\[
\frac{\sinh(\psi(y)-\psi(x))}{y-x} \geq \frac{\psi(y)-\psi(x)}{y-x}. 
\]
Now, fix $x$ and set $\phi(y) := \psi(y)-\psi(x)-(y-x)(\psi'(y) + \psi'(x))/2$. Then it yields that 
\[
\phi'(y) = \psi'(y) -  \frac{1}{2}(\psi'(y) + \psi'(x)) - \frac{1}{2}(y-x)\psi''(y) = \frac{y-x}{2} \left( \frac{\psi'(y) - \psi'(x)}{y-x} - \psi''(y) \right).
\]
Therefore it follows from the mean-value theorem and the monotonicity of $\psi''$ that $\phi'$ is nonnegative for $y > \max\{x, 2\xi-x\}$. 
Thus since $\psi(x)=\psi(2\xi-x)$ and $\psi'(x)=-\psi'(2\xi-x)$, we have $\phi(y) \geq 0$ for any $y>x$ with $\psi(x) \leq \psi(y)$, which implies \eqref{sinh+}.

Furthermore, we note that any probability measure $\mu$ on $[0, \infty)$ whose density $f$ supported on $[0, \infty)$ is non-increasing and satisfies \eqref{sinh+} enjoys that $\mu^*(A) \geq \D^{\flat}(f, [0, \infty))(\theta)$ for given $\theta \in [0, 1]$ and an interval $A \subset [0, \infty)$ with $\mu(A)=\theta$. 
This assertion is also confirmed by the same argument as in Proposition \ref{char_of_min}. 
For instance, the probability measure $\mu_s$ defined by \eqref{aaaaa} for $s \leq 0$ satisfies these properties (although the same result holds for $\mu_s$ with $s \in (0, 1]$, we need additional (but not difficult) discussions since the support of $\mu_s$ is compact).
On the other hand, all log-concave probability measures on $[0, \infty)$ with non-increasing densities do not satisfy \eqref{sinh+}. Indeed, we see that the probability measure whose density is proportional to $e^{-x^3}\mathbf{1}_{[0, \infty)}(x)$ does not satisfy \eqref{sinh+}.

\begin{Thm}\label{thm: minimizer}
Let $\mu$ be a probability measure on $\R$ as in Proposition \ref{char_of_min} and we use the same notations. 
Take an interval $A_\theta \subset \R$ for $\theta \in [0, 1]$ as in Proposition \ref{char_of_min}. 
In addition, we assume that $\mu^*(A_{\theta})$ is concave in $\theta \in [0, 1]$. 
Then we have $\D_{\mu}(\theta)=\mu^*(A_{\theta})$ for every $\theta \in [0, 1]$.
\end{Thm}

\begin{proof}
Fix $\theta_0 \in [0, 1]$ and let $A$ be a Borel subset with $\mu(A)=\theta_0$.
We will show $\mu^*(A)\geq \mu^*(A_{\theta_0})$. 
We may assume $\theta_0 \in (0, 1)$, otherwise the assertion is clear. 
In order to show our assertion, it suffices to prove it for a disjoint union of finite closed intervals $A=\bigcup_{\lambda \in \Lambda} A_{\lambda}$ (by an approximation of compact subsets), where $\Lambda$ is a finite set and $A_{\lambda}$ is a closed interval with $A_{\lambda}\cap A_{\lambda'}= \emptyset$ for any $\lambda \neq \lambda' \in \Lambda$. 
Note that $\mu^*(\bigcup_{\lambda \in \Lambda} A_{\lambda}) = \sum_{\lambda \in \Lambda}\mu^*(A_{\lambda})$.
This is because the $\varepsilon$-dilation of $A$ for small enough $\varepsilon>0$ is the disjoint union of those of $A_{\lambda}$.
Thus, by Proposition \ref{char_of_min}, we obtain 
\[
\mu^*(A) \geq \sum_{\lambda \in \Lambda}\mu^*(A_{\theta_\lambda}), 
\]
where $\theta_{\lambda} := \mu(A_\lambda)$ for $\lambda \in \Lambda$.
Since $\mu^*(A_{\theta})$ is concave in $\theta \in [0, 1]$ and $\theta_0 = \sum_{\lambda \in \Lambda} \theta_{\lambda}$, 
we eventually obtain $\mu^*(A) \geq \mu^*(A_{\theta_0})$. This completes the proof.
\end{proof}

In particular, when a probability measure $\mu$ on $\R$ is centrally symmetric, then Theorem \ref{thm: minimizer} implies that $\D_{\mu} = 2 \D^{\flat}(e^{-\psi}, [0, \infty))$ on $[0, 1]$, where $e^{-\psi}$ is the density of $\mu$. 
The following proposition gives a sufficient condition such that $\mu^*(A_\theta)$ defined in Proposition \ref{char_of_min} is concave on $[0, 1]$. 
\begin{Prop}
Let $f:[0, \infty) \to [0, \infty)$ be $C^1$, non-increasing and log-concave (namely, $(\log f)'' \leq 0$) on its support with $0<\int_0^\infty f(t) ~dt<\infty$. 
Then $\D^{\flat}(f, [0, \infty))$ is concave on $[0, 1]$.
\end{Prop}
\begin{proof}
Let us denote $\D^{\flat}(f, [0, \infty))(\theta)$ by $F(\theta)$ for every $\theta \in [0, 1]$. 
Then we have 
\begin{align*}
F(\theta) = \frac{ f(\alpha(\theta))}{\int_0^\infty f(t) ~dt}\alpha(\theta), 
\end{align*}
where $\alpha(\theta) \geq 0$ is given by 
\begin{align}\label{eq: F}
\theta = \frac{\int_0^{\alpha(\theta)} f(t)~dt}{\int_0^{\infty} f(t)~dt}
\end{align}
for every $\theta \in [0, 1]$.
Since the differentiation in $\theta$ of \eqref{eq: F} yields that 
\[
\int_0^{\infty} f(t)~dt = f(\alpha(\theta)) \alpha'(\theta), 
\]
we have 
\[
F'(\theta) = \frac{ f'(\alpha(\theta))}{\int_0^\infty f(t) ~dt}\alpha'(\theta)\alpha(\theta) + \frac{ f(\alpha(\theta))}{\int_0^\infty f(t) ~dt}\alpha'(\theta)
 = \frac{ f'(\alpha(\theta))}{f(\alpha(\theta)) }\alpha(\theta) + 1.
\]
Hence, the concavity of $F$ on $[0, 1]$ is equivalent to the non-increasing property of the function 
\[
\Phi(x) :=\frac{f'(x)}{f(x)}x, \quad x >0
\]
on the support of $f$.
Since  we see that 
\[
\Phi'(x) = \frac{(f''(x)x + f'(x))f(x) - (f'(x))^2x}{f(x)^2} = (\log f)''(x)x + (\log f)'(x), 
\]
by the log-concavity and non-increasing property of $f$ on its support, we obtain $\Phi' \leq 0$ on the support of $f$. 
Hence $\Phi$ is non-increasing on the support of $f$, and we obtain the desired assertion. 
\end{proof}
As a corollary, we can give an explicit representation of the Gaussian dilation profile on $\R$.
\begin{Cor}
The infimum of the dilation profile of the standard Gaussian measure is attained at a centrally symmetric interval. 
In particular, we have 
\[
\D_{\gamma_1}(\theta) = 2\D^\flat(e^{-t^2/2}, [0, \infty))(\theta) = \frac{4}{\sqrt{2\pi}}e^{-\alpha(\theta)^2/2}\alpha(\theta)
\]
for every $\theta \in [0, 1]$, where $\gamma_1$ is the standard Gaussian measure on $\R$ and $\alpha(\theta) \in [0, \infty]$ is given by 
\begin{align}
\theta = \frac{2}{\sqrt{2\pi}}\int_0^{\alpha(\theta)}e^{-t^2/2} ~dt.
\end{align} 
\end{Cor}

\subsection{Proof of Theorem \ref{MT0}}

In this subsection, we complete the proofs of Theorem \ref{MT0} and Corollary \ref{MT1}. 

\begin{Thm}\label{thm:MT0}
Let $\mu $ be a probability measure on an open interval $I$ with a smooth density $e^{-\psi}$ and satisfy $\Ric_{N}\geq K$ and $|I| \leq D$ for some $K \in \R$, $N \in (-\infty, 1)\cup (1, \infty]$ and $D \in (0, \infty]$. 
Then every interval $A \subset I$ satisfies 
\[
\mu^*(A) \geq \mathscr{D}_{K, N, D}(\mu(A)).
\]
\end{Thm}

\begin{proof}
Fix $\theta \in [0, 1]$ and let $A \subset I$ be an interval with $\mu(A)=\theta$. 
Since we easily see that $\mathscr{D}_{K, N, D}(0) = \mathscr{D}_{K, N, D}(1)=0$, we may consider the assertion only for $\theta \in (0, 1)$. 
We will show that $\mu^*(A) \geq \mathscr{D}_{K, N, D}(\theta)$.
Let $a, b \in \R$ be the endpoints of $A$ with $a<b$. 
By moving $A$ left or right such that $\mu^*(A)$ does not increase with keeping the volume $\theta$ as in Proposition \ref{char_of_min}, we may assume that $A$ satisfies either 
\begin{align}\label{sinh_equation}
\frac{\sinh(\psi(b)-\psi(a))}{b-a} = \frac{\psi'(b)+\psi'(a)}{2}
\end{align}
or $\{a, b\} \cap \partial[\mathrm{supp}(\mu)] \neq \emptyset$, where $\partial[\mathrm{supp}(\mu)]$ means the boundary of $\mathrm{supp}(\mu)$. 

{\bf Case 1.} Suppose that $A$ satisfies \eqref{sinh_equation}. 
By considering the reflection at the origin if necessary, we may also assume that $\psi(b) \geq \psi(a)$. 
Then \eqref{sinh_equation} implies that $-\psi'(b) \leq \psi'(a)$. 
Since $\psi$ is a $(K, N-1)$-convex function by $\Ric_N \geq K$, Lemma \ref{le:convex} and its subsequent discussion yield that 
\begin{align}\label{c}
e^{-\psi(x+t)} \leq e^{-\psi(x)} J_{-\psi'(x), K, N}(t)
\end{align}
for any $x ,t \in \R$ with $x, x+t \in I$.
It follows from \eqref{interval0} and \eqref{c} that for $\varepsilon \in (0, 1)$,  
\begin{align*}
	\mu(A_{\varepsilon}) - \mu(A) 
	&= \int_{a-\varepsilon (b-a)/(1-\varepsilon)}^a e^{-\psi(t)} ~dt +\int_b^{b+\varepsilon (b-a)/(1-\varepsilon)} e^{-\psi(t)} ~dt \\
	&=\int_{-\varepsilon (b-a)/(1-\varepsilon)}^0 e^{-\psi(t+a)} ~dt +\int_0^{\varepsilon (b-a)/(1-\varepsilon)} e^{-\psi(t+b)} ~dt \\
	&\leq e^{-\psi(a)}\int_{-\varepsilon (b-a)/(1-\varepsilon)}^0 J_{-\psi'(a), K, N}(t) ~dt + e^{-\psi(b)} \int_0^{\varepsilon (b-a)/(1-\varepsilon)} J_{-\psi'(b), K, N}(t) ~dt,
\end{align*}
and hence letting $\mu(A_\varepsilon) \to 1$ ($\varepsilon(b-a)/(1-\varepsilon) \to D+a-b$), we obtain 
\begin{align}\label{inequality21}
	1-\theta &\leq e^{-\psi(a)}\int_{-D+b-a}^0 J_{-\psi'(a), K, N}(t) ~dt + e^{-\psi(b)} \int_0^{D-(b-a)} J_{-\psi'(b), K, N}(t) ~dt \notag \\
	&\leq (b-a)^{-1}\mu^*(A) \int_0^{D-(b-a)} J_{\psi'(a), K, N}(t) ~dt,
\end{align}
where we used \eqref{interval} and $-\psi'(b) \leq \psi'(a)$ in the last inequality. 

On the other hand, it follows from \eqref{c} that we obtain 
\begin{align}\label{inequality22}
\theta &= \int_a^b e^{-\psi(t)} ~dt =\int_0^{b-a} e^{-\psi(t+a)} ~dt \notag \\
&\leq e^{-\psi(a)} \int_0^{b-a} J_{-\psi'(a), K, N}(t) ~dt \notag \\
&\leq (b-a)^{-1}\mu^*(A) \int_{-(b-a)}^0 J_{\psi'(a), K, N}(t) ~dt.
\end{align}
Therefore, by \eqref{inequality21} and \eqref{inequality22}, we have 
\begin{align}\label{inequality23}
\mu^*(A) \geq (b-a)\inf_{H \in \R} \max \left\{\frac{1-\theta}{\int_0^{D-(b-a)} J_{H, K, N}(t) ~dt}, \frac{\theta}{\int_{-(b-a)}^0 J_{H, K, N}(t) ~dt} \right\} .
\end{align}

{\bf Case 2.} Suppose that $\{a, b\} \cap \partial[\mathrm{supp}(\mu)] \neq \emptyset$. 
Without loss of generality, we may assume that $\{a, b\} \cap \partial[\mathrm{supp}(\mu)] =\{a\}$.
We remark that in this case, \eqref{interval0} yields $\mu(A_{\varepsilon}) = \mu((a, b+\varepsilon(b-a)/(1-\varepsilon)])$, and hence $\mu^*(A) = e^{-\psi(b)}(b-a)$. 
By \eqref{c}, we have 
\begin{align*}
\mu(A_{\varepsilon}) - \mu(A) &= \int_b^{b+\varepsilon (b-a)/(1-\varepsilon)} e^{-\psi(t)} ~dt \\
& \leq e^{-\psi(b)} \int_0^{\varepsilon (b-a)/(1-\varepsilon)} J_{-\psi'(b), K, N}(t) ~dt.
\end{align*}
Thus, letting $\mu(A_\varepsilon) \to 1$ ($\varepsilon(b-a)/(1-\varepsilon) \to D+a-b$), we have 
\begin{align}\label{inequality4}
1-\theta &\leq e^{-\psi(b)} \int_0^{D-(b-a)} J_{-\psi'(b), K, N}(t) ~dt \notag \\
	&= (b-a)^{-1}\mu^*(A) \int_0^{D-(b-a)} J_{-\psi'(b), K, N}(t) ~dt. 
\end{align}
 On the other hand, 
\begin{align}\label{inequality5}
\theta &= \int_a^b e^{-\psi(t)} ~dt =\int_{-(b-a)}^0 e^{-\psi(t+b)} ~dt \notag \\
&\leq e^{-\psi(b)} \int_{-(b-a)}^0 J_{-\psi'(b), K, N}(t) ~dt \notag \\
&= (b-a)^{-1}\mu^*(A) \int_{-(b-a)}^0 J_{-\psi'(b), K, N}(t) ~dt.
\end{align}
Thus, inequalities \eqref{inequality4} and \eqref{inequality5} also yield \eqref{inequality23}. 

This completes the proof.
\end{proof}

Now, we can prove Theorem \ref{MT0} by Theorem \ref{thm:MT0} and Theorem \ref{localization}.
\begin{proof}[Proof of Theorem \ref{MT0}]
Let $A \subset \M$ be a strongly-convex subset with $\m(A) = \theta_0 \in [0, 1]$. 
We define a function $f$ on $\M$ by $f := \mathbf{1}_A - \theta_0$ which satisfies $\int_\M f ~d\m=0$, where $\mathbf{1}_A$ is the characteristic function on $A$. 
Then by Theorem \ref{localization} for $f$, we have a partition $Q$ of $\M$, a measure $\nu$ on $Q$ and a family $\{\mu_I\}_{I \in Q}$ of probability  measures on $\M$ satisfying (i), (ii) and (iii) in Theorem \ref{localization}.
In particular, (iii) means that $\mu_I(A) = \theta_0$ for $\nu$-almost every $I \in Q$. 
Note also that since $A$ is strongly-convex and $\nu$-almost every $I \in Q$ is a minimizing geodesic, $A\cap I$ is an interval. 
Since $\nu$-almost every $I \in Q$ is open and $(I, |\cdot|, \mu_I)$ satisfies $\Ric_N \geq K$ and $\mathrm{diam} I \leq D$, it follows from Theorem \ref{thm:MT0} that $\mu_I^*(A \cap I) \geq  \mathscr{D}_{K, N, D}(\theta_0)$ holds for $\nu$-almost every $I \in Q$. 
Since $(A \cap I)_{\varepsilon} \subset A_{\varepsilon}$ in $\M$ for any $\varepsilon \in (0, 1)$ and the $\varepsilon$-dilation of $A \cap I$ in $\M$ includes the one in $I$ by the definition of the $\varepsilon$-dilation, we have 
\begin{align*}
\m^*(A) &= \liminf_{\varepsilon \to 0} \frac{\m(A_{\varepsilon}) - \m(A)}{\varepsilon} = \liminf_{\varepsilon \to 0}\int_Q \frac{\mu_I(A_{\varepsilon}) - \mu_I(A)}{\varepsilon} ~d\nu(I) \\
          &\geq \liminf_{\varepsilon \to 0}\int_Q \frac{\mu_I((A\cap I)_{\varepsilon}) - \mu_I(A)}{\varepsilon} ~d\nu(I) 
          \geq \int_Q \mu_I^*(A\cap I) ~d\nu(I) \\
          &\geq \mathscr{D}_{K, N, D}(\theta_0) .
\end{align*}
Hence, we obtain the desired assertion. 
\end{proof}

Next, we prove Corollary \ref{MT1}. 
In order to prove this corollary, we need the following two lemmas. 

\begin{Lem}\label{lem: lemma0}
Let $\mu$ be a probability measure supported on $(a, b)$ ($-\infty<a<b\leq \infty$) whose density is $f$ and let $c \in (a, b)$. 
If $f$ is non-decreasing, then we have $\mu^*((a, c)) \geq \mu((a, c))$. 
On the other hand, if $f$ is non-increasing, then we have $\mu^*((a, c)) \leq \mu((a, c))$.
\end{Lem}

\begin{proof}
Note that $\mu^*((a, c)) = f(c)(c-a)$ follows from direct calculations as in \eqref{interval} (or Case 2 in Theorem \ref{thm:MT0}).
When $f$ is non-decreasing, we easily see that 
\[
f(c)(c-a) \geq \int_a^c f(x) ~dx =\mu((a, c)).
\]
Similarly, when $f$ is non-increasing, we have 
\[
f(c)(c-a) \leq \int_a^c f(x) ~dx =\mu((a, c)), 
\]
and hence we obtain the desired claim.
\end{proof}

\begin{Lem}\label{lemma1}
Let $f :[a, b) \to [0, \infty)$ $(-\infty<a<b\leq \infty)$ be a $C^1$ and integrable function satisfying $f>0$ on $(a, b)$.
If $f'(x)(x-a)/f(x)$ is non-decreasing in $x\in (a, b)$, then for any $\theta \in (0, 1)$, $\D^{\flat}(f, [a, x])(\theta)$ is non-decreasing in $x \in (a, b)$.
Similarly, if $f'(x)(x-a)/f(x)$ is non-increasing in $x \in (a, b)$, then for any $\theta \in (0, 1)$, $\D^{\flat}(f, [a, x])(\theta)$ is non-increasing in $x\in (a, b)$.
\end{Lem}

\begin{proof}
By translation, we may assume that $a =0$. 
For every $\theta \in (0, 1)$, we have 
\begin{align}\label{eq: p1}
	\D^{\flat}(f, [0, x])(\theta) = \frac{f(\alpha(x))}{\int_0^x f(t) ~dt}\alpha(x), 
\end{align}
where $\alpha(x) \in (0, b)$ is given by 
\begin{align}\label{eq: p2}
\theta = \frac{\int_0^{\alpha(x)} f(t) ~dt}{\int_0^x f(t) ~dt}.
\end{align}
Now, the differentiation of \eqref{eq: p2} in $x$ yields that 
\[
\theta f(x) = f(\alpha(x))\alpha'(x), 
\]
and hence we have 
\begin{align*}
\frac{d}{dx}\D^{\flat}(f, [0, x])(\theta) &= \frac{(f'(\alpha(x))\alpha(x) + f(\alpha(x)))\alpha'(x)\int_0^x f(t)~dt - \alpha(x)f(\alpha(x))f(x)}{\left(\int_0^xf(t) ~dt\right)^2} \\
&= \frac{\left((f'(\alpha(x))\alpha(x) + f(\alpha(x)))\int_0^{\alpha(x)} f(t)~dt - \alpha(x)f(\alpha(x))^2\right)f(x)}{f(\alpha(x))\left(\int_0^xf(t) ~dt\right)^2}, 
\end{align*}
where we also used \eqref{eq: p2} in the second equality.
Thus, we see that the claim of $\D^{\flat}(f, [0, x])(\theta)$ being non-decreasing in $x$ is equivalent to 
\begin{align}\label{ineq: p3}
(f'(x)x+f(x))\int_0^x f(t) ~dt -xf(x)^2 \geq 0
\end{align}
for any $x \in (0, b)$.
We can deduce this inequality from the assumption. Indeed, the integration by parts and the non-decreasing property of $f'(x)x/f(x)$ in $x$ yield that 
\begin{align*}
\int_0^x f(t) ~dt  = xf(x) - \int_0^x tf'(t) ~dt \geq xf(x) - \frac{f'(x)x}{f(x)} \int_0^x f(t) ~dt, 
\end{align*}
which implies \eqref{ineq: p3}.

The non-increasing case also follows from a similar argument.
\end{proof}

%
%

\begin{proof}[Proof of Corollary \ref{MT1}.]
Let $(K, N, D) \in \R\times \R \times [0, \infty)$ be  a triple as in Cases 1-9 of Corollaries 1.3 and set $\delta:= K/(N-1)$. 
It suffices to analyze
\begin{align}\label{inequality33}
\Phi(\theta):=(b-a)\inf_{H \in \R} \max \left\{\frac{1-\theta}{\int_0^{D-(b-a)} J_{H, K, N}(t) ~dt}, \frac{\theta}{\int_{-(b-a)}^0 J_{H, K, N}(t) ~dt} \right\}
\end{align} 
for fixed $\theta \in (0, 1)$ and $a, b\in \R$ with $0<b-a<D$, which is derived from \eqref{inequality23}.
When $H$ varies from $-\infty$ to $\infty$, then the first term in the right hand side of \eqref{inequality33} monotonically and continuously (including the value $\infty$) varies from $\infty$ to $0$, and the second term also monotonically and continuously varies from $0$ to $\infty$ (also see \cite[Proposition 3.3]{Mi2}).
Thus, there exists a unique point $H_{\theta} \in \R$ satisfying 
\[
\frac{\int_0^{D-(b-a)} J_{H_{\theta}, K, N}(t) ~dt}{1-\theta} = \frac{\int_{-(b-a)}^0 J_{H_{\theta}, K, N}(t) ~dt}{\theta} = \int_{-(b-a)}^{D-(b-a)} J_{H_{\theta}, K, N}(t) ~dt < \infty.
\] 
Therefore, we have  
\begin{align}\label{inequality_infty}
\Phi(\theta)=\left( \int_{-(b-a)}^{D-(b-a)} J_{H_{\theta}, K, N}(t) ~dt \right)^{-1} (b-a)
\end{align}
with 
\begin{align}\label{theta}
\theta = \frac{\int_{-(b-a)}^0 J_{H_{\theta}, K, N}(t) ~dt}{\int_{-(b-a)}^{D-(b-a)} J_{H_{\theta}, K, N}(t) ~dt}.
\end{align}

{\bf Case 1.} Suppose $N=\infty$, $K > 0$ and $D=\infty$. 
Then we have 
\begin{align}\label{l}
J_{H_{\theta}, K, \infty} (t) = \exp\left(H_{\theta}t -\frac{1}{2}Kt^2\right) = \exp\left(\frac{H_{\theta}^2}{2K} -\frac{K}{2}\left(t-\frac{H_{\theta}}{K}\right)^2 \right).
\end{align}
Thus the right hand side of \eqref{inequality_infty} becomes
\[
	\frac{e^{-H_{\theta}^2/(2K)}}{\int_{-(b-a)-H_{\theta}/K}^{\infty} e^{-Kt^2/2} ~dt} (b-a),
\]
where $H_{\theta}$ satisfies 
\[
	\theta = \frac{\int_{-(b-a)-H_{\theta}/K}^{-H_{\theta}/K} e^{-Kt^2/2} ~dt}{\int_{-(b-a)-H_{\theta}/K}^{\infty} e^{-Kt^2/2} ~dt}
\]
by \eqref{theta}, which implies that 
\[
\Phi(\theta) \geq \inf_{x \in \R} \D^{\flat}(e^{-Kt^2/2}, [x, \infty))(\theta).
\]

{\bf Case 2.} Suppose $N=\infty$, $K \neq 0$ and $D<\infty$. 
As in Case 1, \eqref{l} yields that the right hand side of \eqref{inequality_infty} becomes
\[
	\frac{e^{-H_{\theta}^2/(2K)}}{\int_{-(b-a)-H_{\theta}/K}^{D-(b-a)-H_{\theta}/K} e^{-Kt^2/2} ~dt} (b-a),
\]
where $H_{\theta}$ satisfies 
\[
	\theta = \frac{\int_{-(b-a)-H_{\theta}/K}^{-H_{\theta}/K} e^{-Kt^2/2} ~dt}{\int_{-(b-a)-H_{\theta}/K}^{D-(b-a)-H_{\theta}/K} e^{-Kt^2/2} ~dt}
\]
by \eqref{theta}, which implies that 
\[
\Phi(\theta) \geq \inf_{x \in \R} \D^{\flat}(e^{-Kt^2/2}, [x, x+D])(\theta).
\]

{\bf Case 3.} Suppose $N=\infty$ and $K=0$. Then we have $J_{H_{\theta}, 0, \infty}(t)=e^{H_{\theta} t}$. 
The right hand side of \eqref{inequality_infty} becomes
\begin{align}\label{eq:a}
\frac{e^{H_{\theta}(b-a)}}{\int_0^D e^{H_{\theta}t} ~dt}(b-a), 
\end{align}
and by \eqref{theta}, it holds that 
\begin{align}\label{theta1}
\theta = \frac{\int_{0}^{b-a} e^{H_{\theta}t} ~dt}{\int_{0}^{D} e^{H_{\theta}t} ~dt}.
\end{align}

If $D=\infty$, by the integrability of $J_{H_{\theta}, 0, \infty}$, we see that $H_{\theta}<0$. 
Hence we obtain 
\[
\Phi(\theta) \geq \D^{\flat}(e^{-t}, [0, \infty)) = -(1-\theta)\log(1-\theta),
\]
where we used the scale invariance of $\D^{\flat}(e^{-\lambda t}, [0, \infty))$ for any $\lambda>0$.

Next, suppose $D<\infty$. 
Informally, since the dilation area is scale invariant and $(I, \lambda|\cdot|, \mu)$ satisfies $\Ric_N \geq 0$ and $|I| \leq \lambda D$ for any $\lambda>0$ when $(I, |\cdot|, \mu)$ satisfies $\Ric_N \geq 0$ and $|I| \leq D$, letting $\lambda \to \infty$, we can deduce 
\[
\Phi(\theta) \geq \D^{\flat}(e^{-t}, [0, \infty)) 
\]
(the same argument is true in Cases 5-2 and 8-2 below).
More precisely, 
it follows from \eqref{eq:a} and \eqref{theta1} that we obtain
\begin{align*}
\Phi(\theta) &\geq \inf_{x \in \R} \left\{ \frac{e^{-x \alpha(x)}}{\int_0^D e^{-xt} ~dt} \alpha(x) ~\Bigg{|}~ \frac{\int_0 ^{\alpha(x)} e^{-xt} ~dt }{\int_0^D e^{-xt} ~dt } = \theta \right\} 
= \inf_{x \in \R} \left\{ \left(\theta-\frac{1}{1-e^{-xD}} \right) \log \left( 1-\left( 1-e^{-xD} \right)\theta \right) \right\} \\
&=\lim_{x\to\infty} \left(\theta-\frac{1}{1-e^{-x}} \right) \log \left( 1-\left( 1-e^{-x} \right)\theta \right) 
= -(1-\theta)\log(1-\theta)
= \D^{\flat}(e^{-t}, [0, \infty))(\theta).
\end{align*}

{\bf Case 4.} Suppose $N \in (1, \infty)$ and $K>0$.
In this case, we see that 
\begin{align*}
	J_{H_{\theta}, K, N}(t) &= \left( \cos(\sqrt{\delta} t) + \frac{H_{\theta}}{(N-1)\sqrt{\delta}}\sin(\sqrt{\delta}t)  \right)_+^{N-1} 
		= \left( \frac{\sin(\beta_{\theta} + \sqrt{\delta} t)}{\sin(\beta_{\theta})} \right)_+^{N-1},
\end{align*}
where 
\[
\beta_{\theta} := \cot^{-1}\left(\frac{H_{\theta}}{(N-1)\sqrt{\delta}}\right) \in (0, \pi).  
\]
Thus, the right hand sides of \eqref{inequality_infty} and \eqref{theta} become
\begin{align*}
		\frac{\sin^{N-1}(\beta_{\theta})}{\int_{-(b-a)+\beta_{\theta}/\sqrt{\delta}}^{D-(b-a)+\beta_{\theta}/\sqrt{\delta}} (\sin(\sqrt{\delta} t))_+^{N-1} ~dt}(b-a)
\end{align*}
and 
\[
\theta = \frac{\int_{-(b-a)+\beta_{\theta}/\sqrt{\delta}}^{\beta_{\theta}/\sqrt{\delta}} (\sin(\sqrt{\delta} t))_+^{N-1} ~dt}{\int_{-(b-a)+\beta_{\theta}/\sqrt{\delta}}^{D-(b-a)+\beta_{\theta}/\sqrt{\delta}} (\sin(\sqrt{\delta} t))_+^{N-1} ~dt},
\]
respectively, which imply that
\begin{align*}
	\Phi(\theta) &\geq \inf_{x \in (-D, \pi/\sqrt{\delta})} \D^{\flat} ((\sin(\sqrt{\delta}t))_+^{N-1}, [x, x+D])(\theta).
\end{align*}
Now, the function $f(t):=\sin^{N-1}(\sqrt{\delta}t)$ satisfies $f'(t)t/f(t)=(N-1)\sqrt{\delta}\cot(\sqrt{\delta}t)t$, which is strictly decreasing in $t \in (0, \pi/ \sqrt{\delta})$. 
Thus, by Lemma \ref{lemma1}, we have 
\[
 \Phi(\theta)\geq \inf_{x \in [0, \pi/\sqrt{\delta})} \D^{\flat} (\sin^{N-1}(\sqrt{\delta}t), [x, \min\{x+D, \pi/\sqrt{\delta}\}])(\theta).
\]

{\bf Case 5-1.} Suppose $N \in (1 ,\infty)$, $K=0$ and $D=\infty$. 
In this case, we have $J_{H_{\theta}, 0,N}(t) =(1+H_{\theta}t/(N-1))_+^{N-1}$, and hence $H_{\theta}$ is necessarily negative by the integrability of $J_{H_{\theta}, 0,N}$.
Thus, the right hand side of \eqref{inequality_infty} becomes 
\[
\frac{ (-(N-1)/H_{\theta})^{N-1} }{ \int_{-(b-a)+(N-1)/H_{\theta}}^{0} (-t)^{N-1} ~dt } (b-a),
\]
and it follows from \eqref{theta} that 
\[
\theta = \frac{ \int_{-(b-a)}^{0} (1+H_{\theta}t/(N-1))_+^{N-1} ~dt }{ \int_{-(b-a)}^{\infty} (1+H_{\theta}t/(N-1))_+^{N-1} ~dt } = \frac{ \int_{-(b-a)+(N-1)/H_{\theta}}^{(N-1)/H_{\theta}} (-t)^{N-1} ~dt }{ \int_{-(b-a)+(N-1)/H_{\theta}}^{0} (-t)^{N-1} ~dt }.
\]
Therefore, we obtain 
\begin{align*}
\Phi(\theta) \geq \inf_{x < 0} \D^{\flat}( (-t)^{N-1}, [x, 0])(\theta).
\end{align*}
It is easy to confirm that $\D^{\flat}( (-t)^{N-1}, [x, 0])$ is independent of $x$ by the scale invariance, and hence it yields
\[
\Phi(\theta) \geq \D^{\flat}( (-t)^{N-1}, [-1, 0])(\theta) = -N(1-\theta-(1-\theta)^{1-1/N}).
\]

{\bf Case 5-2.} Suppose $N \in (1, \infty)$, $K=0$ and $D<\infty$. 
Since we have $J_{H_{\theta}, 0,N}(t) =(1+H_{\theta}t/(N-1))_+^{N-1}$, 
the right hand side of \eqref{inequality_infty} becomes 
\begin{align*}
	\begin{cases}
		\frac{ ((N-1)/H_{\theta})^{N-1} }{ \int_{\xi_1}^{\xi_2} (t)_+^{N-1} ~dt } (b-a) &\text{if $H_{\theta}>0$}, \\
		\frac{b-a}{D} &\text{if $H_{\theta}=0$}, \\
		\frac{ (-(N-1)/H_{\theta})^{N-1} }{ \int_{\xi_1}^{\xi_2} (-t)_+^{N-1} ~dt } (b-a) &\text{if $H_{\theta}<0$},
	\end{cases}
\end{align*}
and it follows from \eqref{theta} that  
\begin{align*}
\theta = \frac{ \int_{\xi_1}^{\xi_3} (H_{\theta}t /(N-1))_+^{N-1} ~dt }{ \int_{\xi_1}^{\xi_2}(H_{\theta}t /(N-1))_+^{N-1} ~dt } =
	\begin{cases}
	 \frac{ \int_{\xi_1}^{\xi_3} (t)_+^{N-1} ~dt }{ \int_{\xi_1}^{\xi_2} (t)_+^{N-1} ~dt }  &\text{if $H_{\theta}>0$}, \\
	 \frac{b-a}{D} &\text{if $H_{\theta}=0$}, \\
	 \frac{ \int_{\xi_1}^{\xi_3} (-t)_+^{N-1} ~dt }{ \int_{\xi_1}^{\xi_2} (-t)_+^{N-1} ~dt } &\text{if $H_{\theta}<0$},
	\end{cases}
\end{align*}
where $\xi_1 := -(b-a)+(N-1)/H_{\theta}$, $\xi_2:= D-(b-a)+(N-1)/H_{\theta}$ and $\xi_3 := (N-1)/H_{\theta}$.
Therefore, we obtain 
\begin{align*}
\Phi(\theta) \geq \min \left\{
	\begin{array}{ll}
		\inf_{x \in (-D, \infty)} \D^{\flat}( (t)_+^{N-1}, [x, x+D])(\theta), \\
		\D^{\flat}( 1, [0, D])(\theta), \\
		\inf_{x \in (-\infty, 0)} \D^{\flat}( (-t)_+^{N-1}, [x, x+D])(\theta)
	\end{array}
\right\}.
\end{align*}
By Lemma \ref{lem: lemma0}, it holds that 
\begin{align*}
	\inf_{x \in (-D, \infty)} \D^{\flat}( (t)_+^{N-1}, [x, x+D])(\theta) \geq \theta \geq \inf_{x \in (-\infty, 0)} \D^{\flat}( (-t)_+^{N-1}, [x, x+D])(\theta).
\end{align*}
We can also confirm that $\D^{\flat}( 1, [0, D])(\theta) = \theta$ by direct calculations.
Hence, we obtain that  
\begin{align*}
\Phi(\theta) &\geq \inf_{x \in (-\infty, 0)} \D^{\flat}( (-t)_+^{N-1}, [x, x+D])(\theta) \\
	&= \inf_{x \in (-\infty, 0)} \D^{\flat}( (-t)^{N-1}, [x, \min\{x+D, 0\}])(\theta).
\end{align*}
Note that the function $f(t) := (-t)^{N-1}$ satisfies $f'(t)(t-\xi)/f(t) = (N-1)(1-\xi/t)$ for any $\xi<0$, which is strictly decreasing on $t<0$ independently of $\xi$. 
Thus for any $-\infty <y<z \leq 0$, it follows from scale transformation and Lemma \ref{lemma1} that 
\[
	\D^{\flat}( (-t)^{N-1}, [y,z])(\theta) = \D^{\flat}( (-t)^{N-1}, [-D, zD/|y|])(\theta)
	\geq \D^{\flat}( (-t)^{N-1}, [-D, 0])(\theta).
\]
Hence, we obtain 
\[
\inf_{x \in (-\infty, 0)} \D^{\flat}( (-t)^{N-1}, [x, \min\{x+D, 0\}])(\theta)
=  \inf_{x \in (-\infty, 0)} \D^{\flat}( (-t)^{N-1}, [x, 0])(\theta)
= \D^{\flat}( (-t)^{N-1}, [-1, 0])(\theta).
\]

{\bf Case 6.} Suppose $N \in (1, \infty)$, $K<0$ and $D<\infty$.
In this case, we see that 
\begin{align*}
	J_{H_{\theta}, K, N}(t) &= \left( \cosh(\sqrt{-\delta} t) + \frac{H_{\theta}}{(N-1)\sqrt{-\delta}}\sinh(\sqrt{-\delta}t)  \right)_+^{N-1} \\
	&= 
	\begin{cases}
		\left( \frac{\sinh(\beta_{\theta} + \sqrt{-\delta} t)}{\sinh(\beta_{\theta})} \right)_+^{N-1} &\text{if } | \frac{H_{\theta}}{(N-1)\sqrt{-\delta}} | >1, \\
		\left( \frac{\cosh(\beta_{\theta} + \sqrt{-\delta} t)}{\cosh(\beta_{\theta})} \right)^{N-1} &\text{if } | \frac{H_{\theta}}{(N-1)\sqrt{-\delta}} | <1, \\
		e^{\sqrt{-\delta}(N-1)t} &\text{if } \frac{H_{\theta}}{(N-1)\sqrt{-\delta}} = 1, \\
		e^{-\sqrt{-\delta}(N-1)t} &\text{if } \frac{H_{\theta}}{(N-1)\sqrt{-\delta}}=-1, \\
	\end{cases}  
\end{align*}
where 
\begin{align*}
\beta_{\theta} := 
	\begin{cases}
		\coth^{-1}(\frac{H_{\theta}}{(N-1)\sqrt{-\delta}}) &\text{if } |\frac{H_{\theta}}{(N-1)\sqrt{-\delta}}| >1, \\
		\tanh^{-1}(\frac{H_{\theta}}{(N-1)\sqrt{-\delta}}) &\text{if } |\frac{H_{\theta}}{(N-1)\sqrt{-\delta}}| <1. 
	\end{cases}
\end{align*}
When $|H_{\theta}/(N-1)| \neq \sqrt{-\delta}$, 
the right hand side of \eqref{inequality_infty} becomes
\begin{align*}
	\begin{cases}
		\frac{\sinh^{N-1}(\beta_{\theta})}{\int_{\xi_1}^{\xi_2} (\sinh(\sqrt{-\delta} t))_+^{N-1} ~dt}(b-a) &\text{if }  \frac{H_{\theta}}{(N-1)\sqrt{-\delta}}  >1, \\
		\frac{\sinh^{N-1}(-\beta_{\theta})}{\int_{\xi_1}^{\xi_2} (\sinh(-\sqrt{-\delta} t))_+^{N-1} ~dt}(b-a) &\text{if }  \frac{H_{\theta}}{(N-1)\sqrt{-\delta}}  <-1, \\
		\frac{\cosh^{N-1}(\beta_{\theta})}{\int_{\xi_1}^{\xi_2} \cosh^{N-1}(\sqrt{-\delta} t) ~dt}(b-a) &\text{if } | \frac{H_{\theta}}{(N-1)\sqrt{-\delta}} | <1, 
	\end{cases}
\end{align*}
and the right hand side of \eqref{theta} becomes
\begin{align*}
	\begin{cases}
		\frac{\int_{\xi_1}^{\xi_3} (\sinh(\sqrt{-\delta} t))_+^{N-1} ~dt}{\int_{\xi_1}^{\xi_2} (\sinh(\sqrt{-\delta} t))_+^{N-1} ~dt} &\text{if }  \frac{H_{\theta}}{(N-1)\sqrt{-\delta}}  >1, \\
		\frac{\int_{\xi_1}^{\xi_3} (\sinh(-\sqrt{-\delta} t))_+^{N-1} ~dt}{\int_{\xi_1}^{\xi_2} (\sinh(-\sqrt{-\delta} t))_+^{N-1} ~dt} &\text{if }  \frac{H_{\theta}}{(N-1)\sqrt{-\delta}}  <-1, \\
		\frac{\int_{\xi_1}^{\xi_3} \cosh^{N-1}(\sqrt{-\delta} t) ~dt}{\int_{\xi_1}^{\xi_2} \cosh^{N-1}(\sqrt{-\delta} t) ~dt} &\text{if } | \frac{H_{\theta}}{(N-1)\sqrt{-\delta}} | <1, 
	\end{cases}
\end{align*}
where $\xi_1 := -(b-a)+\beta_{\theta}/\sqrt{-\delta}$, $\xi_2 := D-(b-a)+\beta_{\theta}/\sqrt{-\delta}$ and $\xi_3 :=\beta_{\theta}/\sqrt{-\delta}$.
Therefore, combining these with the argument in Case 3 for $|H_{\theta}/(N-1)| = \sqrt{-\delta}$ , it follows from Lemma \ref{lem: lemma0} that 
\begin{align*}
\Phi(\theta) &\geq \min \left\{
	\begin{array}{ll}
		\inf_{x \in (-D, \infty)} \D^{\flat}( (\sinh(\sqrt{-\delta}t))_+^{N-1}, [x, x+D])(\theta), \\
		\inf_{x \in (-\infty, 0)} \D^{\flat}( (\sinh(-\sqrt{-\delta}t))_+^{N-1}, [x, x+D])(\theta), \\
		\inf_{x \in \R} \D^{\flat}( \cosh^{N-1}(\sqrt{-\delta}t), [x, x+D])(\theta), \\
		\D^{\flat}( e^{-\sqrt{-\delta}(N-1)t}, [0, D])(\theta),\\
		\D^{\flat}( e^{\sqrt{-\delta}(N-1)t}, [0, D])(\theta),
	\end{array}
\right\} \\
	&= \min \left\{
	\begin{array}{ll}
		\inf_{x \in (-\infty, 0)} \D^{\flat}( \sinh^{N-1}(-\sqrt{-\delta}t), [x, \min \{ x+D, 0\}])(\theta), \\
		\inf_{x \in \R} \D^{\flat}( \cosh^{N-1}(\sqrt{-\delta}t), [x, x+D])(\theta), \\
		\D^{\flat}( e^{-\sqrt{-\delta}(N-1)t}, [0, D])(\theta)
	\end{array}
\right\}.
\end{align*}

{\bf Case 7.} Suppose $N \in (-\infty, 0]$ and $K>0$.
Then we have 
\begin{align*}
	J_{H_{\theta}, K, N}(t) &= \left( \cosh(\sqrt{-\delta} t) + \frac{H_{\theta}}{(N-1)\sqrt{-\delta}}\sinh(\sqrt{-\delta}t)  \right)_+^{N-1} \\
	&= 
	\begin{cases}
		\left( \frac{\sinh(\beta_{\theta} + \sqrt{-\delta} t)}{\sinh(\beta_{\theta})} \right)_+^{N-1} &\text{if }  |\frac{H_{\theta}}{(N-1)\sqrt{-\delta}}|  >1, \\
		\left( \frac{\cosh(\beta_{\theta} + \sqrt{-\delta} t)}{\cosh(\beta_{\theta})} \right)^{N-1} &\text{if } | \frac{H_{\theta}}{(N-1)\sqrt{-\delta}} | <1, \\
		e^{\sqrt{-\delta}(N-1)t} &\text{if } \frac{H_{\theta}}{(N-1)\sqrt{-\delta}} = 1, \\
		e^{-\sqrt{-\delta}(N-1)t} &\text{if } \frac{H_{\theta}}{(N-1)\sqrt{-\delta}} = -1,
	\end{cases}  
\end{align*}
where 
\begin{align*}
\beta_{\theta} := 
	\begin{cases}
		\coth^{-1}(\frac{H_{\theta}}{(N-1)\sqrt{-\delta}}) &\text{if } |\frac{H_{\theta}}{(N-1)\sqrt{-\delta}}| >1, \\
		\tanh^{-1}(\frac{H_{\theta}}{(N-1)\sqrt{-\delta}}) &\text{if } |\frac{H_{\theta}}{(N-1)\sqrt{-\delta}}| <1, 
	\end{cases}
\end{align*}
and we exclude the case $\frac{H_{\theta}}{(N-1)\sqrt{-\delta}} \leq-1$ when $D=\infty$.

When $D=\infty$, by the same argument as in Case 6, we obtain 
\begin{align*}
\Phi(\theta) \geq \min \left\{
	\begin{array}{ll}
		\inf_{x >0} \D^{\flat}( \sinh^{N-1}(\sqrt{-\delta}t), [x, \infty))(\theta), \\
		\inf_{x \in \R} \D^{\flat}( \cosh^{N-1}(\sqrt{-\delta}t), [x, \infty))(\theta), \\
		\D^{\flat}( e^{-t}, [0, \infty))(\theta)
	\end{array}
\right\}.
\end{align*}
Similarly, when $D<\infty$, it holds that by Lemma \ref{lem: lemma0},
\begin{align*}
\Phi(\theta) &\geq \min \left\{
	\begin{array}{ll}
		\inf_{x >0} \D^{\flat}( \sinh^{N-1}(\sqrt{-\delta}t), [x, x+D])(\theta), \\
		\inf_{x <0} \D^{\flat}( \sinh^{N-1}(-\sqrt{-\delta}t), [x-D, x])(\theta), \\
		\inf_{x \in \R} \D^{\flat}( \cosh^{N-1}(\sqrt{-\delta}t), [x, x+D])(\theta), \\
		\D^{\flat}( e^{-\sqrt{-\delta}(N-1)t}, [0, D])(\theta),\\
		\D^{\flat}( e^{\sqrt{-\delta}(N-1)t}, [0, D])(\theta)
	\end{array}
\right\}\\
	&= \min \left\{
	\begin{array}{ll}
		\inf_{x >0} \D^{\flat}( \sinh^{N-1}(\sqrt{-\delta}t), [x, x+D])(\theta), \\
		\inf_{x \in \R} \D^{\flat}( \cosh^{N-1}(\sqrt{-\delta}t), [x, x+D])(\theta), \\
		\D^{\flat}( e^{\sqrt{-\delta}(N-1)t}, [0, D])(\theta)
	\end{array}
\right\}.
\end{align*} 

{\bf Case 8-1.} Suppose $N \in (-\infty, 0)$, $K=0$ and $D=\infty$. 
In this case, we have $J_{H_{\theta}, 0,N}(t) =(1+H_{\theta}t/(N-1))_+^{N-1}$, and $H_{\theta}$ is necessarily negative by the integrability of $J_{H_{\theta}, 0,N}$.
Thus, 
the right hand side of \eqref{inequality_infty} becomes 
\[
\frac{ ((N-1)/H_{\theta})^{N-1} }{ \int_{-(b-a)+(N-1)/H_{\theta}}^{\infty} (t)_+^{N-1} ~dt } (b-a), 
\]
and it follows from \eqref{theta} that  
\[
\theta = \frac{ \int_{-(b-a)}^{0} (1+H_{\theta}t/(N-1))_+^{N-1} ~dt }{ \int_{-(b-a)}^{\infty} (1+H_{\theta}t/(N-1))_+^{N-1} ~dt } = \frac{ \int_{-(b-a)+(N-1)/H_{\theta}}^{(N-1)/H_{\theta}} (t)_+^{N-1} ~dt }{ \int_{-(b-a)+(N-1)/H_{\theta}}^{\infty} (t)_+^{N-1} ~dt }.
\]
Therefore, we obtain 
\begin{align*}
\Phi(\theta) \geq \inf_{x > 0} \D^{\flat}( t^{N-1}, [x, \infty))(\theta).
\end{align*}
We easily see that $\D^{\flat}( t^{N-1}, [x, \infty))$ is independent of $x>0$ by the scale invariance, and hence it yields
\[
\Phi(\theta) \geq \D^{\flat}( t^{N-1}, [1, \infty))(\theta) = -N(1-\theta-(1-\theta)^{1-1/N}).
\]

{\bf Case 8-2.} Suppose $N \in (-\infty, 0)$, $K=0$ and $D<\infty$. 
Since we have $J_{H_{\theta}, 0,N}(t) =(1+H_{\theta}t/(N-1))_+^{N-1}$, 
by the same argument as in Case 5-2 and Lemma \ref{lem: lemma0}, we obtain 
\begin{align*}
\Phi(\theta) \geq \min \left\{
	\begin{array}{ll}
		\inf_{x >0} \D^{\flat}( t^{N-1}, [x, x+D])(\theta), \\
		\inf_{x <0} \D^{\flat}( (-t)^{N-1}, [x-D, x])(\theta), \\
		\D^{\flat}( 1, [0, D])(\theta)
	\end{array}
\right\}
	= \inf_{x >0} \D^{\flat}( t^{N-1}, [x, x+D])(\theta).
\end{align*}
Note that the function $f(t) := t^{N-1}$ satisfies $f'(t)(t-\xi)/f(t) = (N-1)(1-\xi/t)$ for any $\xi>0$, which is strictly decreasing on $t>0$ independently of $\xi$. 
Thus for any $0 <y<z < \infty$, it follows from Lemma \ref{lemma1} and scale transformation that, putting $\bar{z}:=\max\{z, y+D\}$, 
\begin{align*}
	\D^{\flat}( t^{N-1}, [y,z])(\theta) &\geq \D^{\flat}( t^{N-1}, [y, \bar{z}])(\theta)\\
	&= \D^{\flat}( t^{N-1}, [yD/(\bar{z}-y), \bar{z}D/(\bar{z}-y)])(\theta)\\
	&\geq \inf_{x >0} \D^{\flat}( t^{N-1}, [x, x+D])(\theta), 
\end{align*}
which implies that $\inf_{x >0} \D^{\flat}( t^{N-1}, [x, x+D])(\theta)$ is independent of $D$ and therefore it holds that 
\[
\inf_{x >0} \D^{\flat}( t^{N-1}, [x, x+D])(\theta) 
=\inf_{x >0} \D^{\flat}( t^{N-1}, [x, \infty))(\theta).
\]
Hence, we obtain by scale transformation 
\[
\Phi(\theta) \geq \inf_{x >0} \D^{\flat}( t^{N-1}, [1, \infty))(\theta) = -N(1-\theta-(1-\theta)^{1-1/N}).
\]

{\bf Case 9.} Suppose $N \in (-\infty, 0]$, $K<0$ and $D<\pi/\sqrt{\delta}$. 
In this case, we see that 
\begin{align*}
	J_{H_{\theta}, K, N}(t) &= \left( \cos(\sqrt{\delta} t) + \frac{H_{\theta}}{(N-1)\sqrt{\delta}}\sin(\sqrt{\delta}t)  \right)_+^{N-1} 
		= \left( \frac{\sin(\beta_{\theta} + \sqrt{\delta} t)}{\sin(\beta_{\theta})} \right)_+^{N-1},
\end{align*}
where 
\[
\beta_{\theta} := \cot^{-1}\left(\frac{H_{\theta}}{(N-1)\sqrt{\delta}}\right) \in (0, \pi).  
\]
Thus, by the same argument as in Case 4, we obtain 
\[ 
\Phi(\theta) \geq \inf_{x \in (0, \pi/\sqrt{\delta}-D)} \D^{\flat} (\sin^{N-1}(\sqrt{\delta}t), [x, x+D]).
\]

Hence, we obtain the desired assertion.
\end{proof}

\begin{Rem}\label{rem: rem1}
\begin{itemize}
\item[(1)] When $N \in (1, \infty)$, $K<0$ and $D=\infty$, we see that the function
\[
\R \ni H \mapsto \int_0^\infty J_{H, K, N}(t) ~dt \in [0, \infty]
\]
is not continuous. Indeed, when $H = -(N-1)\sqrt{-\delta}$, then we have 
$J_{H, K, N}(t) = \exp(-\sqrt{-\delta}(N-1)t)$, and thus 
\[
\int_0^\infty J_{H, K, N}(t) ~dt < \infty. 
\]
On the other hand, when $|H| < (N-1)\sqrt{-\delta}$, we have 
\[
J_{H, K, N}(t)=\left( \frac{\cosh(\beta_{\theta} + \sqrt{-\delta} t)}{\cosh(\beta_{\theta})} \right)^{N-1}
\]
(see the proof of Case 6 for $\beta_{\theta})$, and hence 
\[
\int_0^\infty J_{H, K, N}(t) ~dt = \infty.
\]
These properties imply the discontinuity of $\int_0^\infty J_{H, K, N}(t) ~dt$ in $H$.
Therefore, we excluded this case from Case 6 in Corollary \ref{MT1}.
\item[(2)] When $N=0$ and $K=0$, we  see that $\mathscr{D}_{0, 0, D}=0$.
Indeed, this follows from $J_{H, 0, 0} \equiv \infty$ for all $H \in \R$ when $D=\infty$. 
When $D<\infty$, we can also reduce this claim via the same argument as the proof of Case 8-2 above.
Hence we excluded this case from Case 8 in Corollary \ref{MT1}.
\item[(3)] E. Milman also discussed the case $N\in (0, 1)$, $K>0$ and $D=\infty$ for the isoperimetric profile in \cite{Mi2}. 
However in our setting, every $(a, b) \in \Delta_D$ satisfies $a<\infty$, and hence the function 
\[
\R \ni H \mapsto \int_{-a}^0 J_{H, K, N}(t) ~dt \in [0, \infty]
\]
is not continuous in this case (see \cite[Proposition 3.3]{Mi2}). 
\item[(4)] We emphasize that when $K=0$ and $N \in (-\infty, 0)\cup(1, \infty]$, we can completely recover \eqref{b} for any geodesically-convex $n$-dimensional weighted Riemannian manifold. 
For this purpose, we need to prove Theorem \ref{MT0} for any Borel subset. 
By the same argument as in Theorem \ref{MT0} via the needle decomposition, we may consider only the 1-dimensional case. 
Since $\D_{0, N, \infty}$ (which coincides with the right hand side of \eqref{b}) is concave on $[0, 1]$, 
we can eventually reduce the 1-dimensional problem for a Borel subset to the one for an interval. 
However, this assertion is exactly proved in Theorem \ref{MT0}. 
Finally, note that the above argument is also applied to other cases if $\D_{K, N, D}$ is concave.
\end{itemize}
\end{Rem}

\section{Estimates for $\varepsilon$-dilation sets under some regularities}\label{4}

In this section, we consider the dilation inequalities associated with $\varepsilon \in (0, 1)$.
Given $K \in \R$, $N \in (-\infty, 0] \cup (1, \infty]$ and $D \in (0, \infty]$ in Cases 1-9 in Corollary \ref{MT1}, 
let $\D_{K, N, D}$ be the function defined in Corollary \ref{MT1}. 

Firstly, we describe an idea to establish our assertion. 
Let $f:[0, \infty) \to [0, \infty)$ be a $C^1$ function supported on $[0, a_f)$ for some $a_f \in (0, \infty]$ with $\int_0^\infty f(t) ~dt=1$ and $f(0)=1$ (we may assume this condition by scaling). 
We define functions $F: [0, \infty) \to (0, 1)$ and $I: [0, 1] \to [0, \infty)$ by $F(x):=\int_0^x f(t) ~dt$ for $x \in [0, \infty)$ and $I(\theta) := f(F^{-1}(\theta))$ for $\theta \in [0, 1]$, respectively, 
where $F^{-1}$ is the inverse function of $F$. 
In general, it is well-known that $f$ can be recovered by $I$ via
\begin{align}\label{isoperimetric profile}
F^{-1}(\theta) = \int_0^{\theta} \frac{1}{I(t)} ~dt
\end{align}
for any $\theta \in (0, 1)$ since $(F^{-1})'=1/I$ on $(0, 1)$.
Similarly, we can construct the density $f$ from $\D(f, [0, \infty))$. 
For simplicity, let us denote $\D(f, [0, \infty))(\theta)$ by $J(\theta)$ for every $\theta \in [0, 1]$.
By the definition, we have $J = IF^{-1}$ on $[0, 1]$. 
Thus, we see that 
\[
I'=\left(\frac{J}{F^{-1}}\right)' = \frac{J'-1}{F^{-1}} 
\]
on $[0, 1]$. Hence, putting 
\[
\widetilde{J} := \frac{J'-1}{J}  , 
\]
we obtain $\widetilde{J}=I'/I =(\log I)'$,  
which yields that  for any $\theta \in (0, 1)$, 
\[
I(\theta) =\exp\left( \int_0^{\theta} \widetilde{J}(s) ~ds \right). 
\]
Combining this equality with \eqref{isoperimetric profile}, we obtain for any $\theta \in (0, 1)$, 
\[
F^{-1}(\theta)=\int_0^{\theta} \exp\left( -\int_0^{t} \widetilde{J}(s) ~ds \right) ~dt.
\]
Therefore, we can determine the function $f$ from $J$. 
For the dilation inequality associated with $\varepsilon$ below, we use similar functions constructed above via Corollary \ref{MT1}.

Now, given a triple $(K, N, D)$, we denote 
\[
I_{K, N, D}(\theta) := \exp\left( \int_0^\theta \widetilde{\D}_{K, N, D} (s) ~ds\right)
\] 
and 
\begin{align}\label{D_tilde}
F_{K, N, D}^{-1} (\theta) := \int_0^{\theta} \frac{1}{I_{K, N, D}(t)} ~dt
\end{align}
for $\theta \in [0, 1]$, where $\widetilde{\D}_{K, N, D}$ is given by for $s \in (0, 1)$, 
\begin{align*}
\widetilde{\D}_{K, N, D}(s) := \frac{(\D_{K, N, D})'(s) -1}{\D_{K, N, D}(s)}.
\end{align*}
In order to ensure the existence of \eqref{D_tilde}, we assume the following regularities.  

\medskip \noindent
{\bf Assumption (A).} We say that a triple $(K, N, D)$ 
satisfies Assumption (A) if $\D_{K, N, D} \in C([0, 1]) \cap C^1((0, 1))$ and $\lim_{\theta\to0}\widetilde{\D}_{K, N, D}(\theta)$ exists. 
\medskip

When $K=0$, by Corollary \ref{MT1}, $(0, N, D)$ satisfies Assumption (A) for any $N \in (-\infty, 0) \cup (1, \infty]$ and $D \in (0, \infty]$.
More precisely, these cases yield the concavity of $\D_{0, N, D}$. 
The author also expects the concavity of $\D_{1, N, D}$ with $N>1$, in particular $\D_{1, \infty, \infty}$, since 
the corresponding isoperimetric profiles satisfy the concavity. 

We remark that Assumption (A) implies that $\lim_{\theta\to0}(\D_{K, N, D})'(\theta) =1$, and hence 
\begin{align}\label{limit}
\lim_{\theta\to0} \frac{F_{K, N, D}^{-1} (\theta)}{\D_{K, N, D}(\theta)} = 1.
\end{align}
Note also that  $F_{K, N, D}^{-1}(\theta)$ is continuous and strictly increasing in $\theta \in [0, 1]$, and hence we have its inverse function $F_{K, N, D}: [0, F_{K, N, D}^{-1,\infty}] \to [0, 1]$ which is also continuous and strictly increasing, where we denote $F_{K, N, D}^{-1, \infty} := \lim_{\theta \to 1}F_{K, N, D}^{-1}(\theta) \in (0, \infty]$. 

The following assertion is the main theorem in this section.

\begin{Thm}\label{thm: MT2}
Let $(\M, g, \m)$ be a geodesically-convex $n$-dimensional weighted Riemannian manifold satisfying $\m(\M)=1$, $\Ric_N\geq K$ and $\mathrm{diam}\M \leq D$ for some $K\in \R$, $N \in (-\infty, 0]\cup [n, \infty]$ and $D \in (0, \infty]$ ($N\neq 1$ when $n=1$).
Assume that a triple $(K, N, D)$ satisfies Assumption (A). 
Then for any $\varepsilon \in (0, 1)$ and for any strongly-convex subset $A \subset \M$ with $\m(A) < F_{K, N, D}((1-\varepsilon) F_{K, N, D}^{-1, \infty})$, we have 
\[
\m(A_{\varepsilon}) \geq F_{K, N, D}\left(\frac{1}{1-\varepsilon} F_{K, N, D}^{-1}(\m(A))\right). 
\]
\end{Thm}

For simplicity, let us denote $F_{K, N, D}((1-\varepsilon) F_{K, N, D}^{-1, \infty})$ by $F_{K, N, D, \varepsilon}^{\infty}$.
We remark that $F_{K, N, D, \varepsilon}^{\infty}=1$ for any $\varepsilon \in (0, 1)$ when $F_{K, N, D}^{-1, \infty} = \infty$. 
Before proving this theorem, we note that Theorem \ref{thm: MT2} can partially recover Theorem \ref{Klartag}.
Indeed, in the case of  $N \in (-\infty, 0)\cup (1, \infty)$, $K=0$ and $D =\infty$, we see that for any $s \in (0, 1)$, 
\begin{align*}
\widetilde{\D}_{0, N, \infty} (s) = \frac{-N(-1+(1-1/N)(1-s)^{-1/N}) -1}{-N(1-s-(1-s)^{1-1/N})} = -\frac{N-1}{N}\cdot\frac{1}{1-s},
\end{align*}
and hence we have for any $\theta \in [0, 1]$, 
\begin{align*}
I_{0, N, \infty}(\theta) =\exp\left( \int_0^\theta \widetilde{\D}_{0, N, \infty} (s) ~ds \right) 
= \exp\left(-\frac{N-1}{N} \int_0^\theta \frac{1}{1-s}  ~ds\right)
= (1-\theta)^{(N-1)/N}
\end{align*}
and
\begin{align*}
	F_{0, N, \infty}^{-1} (\theta) = \int_0^{\theta} \frac{1}{I_{0, N, \infty}(t)} ~dt 
		= \int_0^{\theta} (1-t)^{-(N-1)/N} ~dt = N - N(1-\theta)^{1/N}.
\end{align*}
Thus, $F_{0, N. \infty}^{-1, \infty} = N$ if $N \in (1, \infty)$ and $\infty$ if $N \in (-\infty, 0)$, and we obtain for any $x \in [0, F_{0, N, \infty}^{-1, \infty})$, 
\begin{align*}
F_{0, N, \infty}(x) = 1- \left( 1- \frac{x}{N} \right)^N. 
\end{align*}
Therefore, we have for any $\theta \in [0, F_{0, N, \infty, \varepsilon}^{\infty})$, 
\begin{align*}
	F_{0, N, \infty}\left(\frac{1}{1-\varepsilon} F_{0, N, \infty}^{-1}(\theta)\right) = 1- \left( 1- \frac{N - N(1-\theta)^{1/N}}{N(1-\varepsilon)} \right)^N
		= 1- \left( \frac{(1-\theta)^{1/N} - \varepsilon}{1-\varepsilon} \right)^N.
\end{align*}
The same argument applies to $N=\infty$ and we obtain $F_{0, \infty. \infty}^{-1, \infty} = \infty$ and 
\[
F_{0, \infty, \infty}\left(\frac{1}{1-\varepsilon} F_{0, \infty, \infty}^{-1}(\theta)\right) = 1-(1-\theta)^{1/(1-\varepsilon)}.
\]
In addition, we see that 
\[
	F_{0, N, \infty, \varepsilon}^{\infty} = 
		\begin{cases}
			1-\varepsilon^N &\text{if $N \in (1, \infty)$}, \\
			1 &\text{if $N \in (-\infty, 0) \cup \{\infty \}$}.
		\end{cases}
\]

\begin{Rem}
More precisely, when $K=0$, we can completely recover Theorem \ref{Klartag} by combining the subsequent arguments in this paper with Remark \ref{rem: rem1}(4) and the decreasing rearrangement used in \cite{NSV} and \cite{BN}.
\end{Rem}

We again use the needle decomposition to prove Theorem \ref{thm: MT2}.
%
Thus, similarly to the proof of Theorem \ref{MT0} via Theorem \ref{localization}, we consider only the 1-dimensional problem of Theorem \ref{thm: MT2}. 
In order to prove the 1-dimensional problem, we need the followings. 
\begin{Prop}\label{prop: concavity2}
Let $K \in \R$, $N \in (-\infty, 0]\cup (1, \infty]$ and $D \in (0, \infty]$, and we assume that a triple $(K, N, D)$ satisfies Assumption (A). Then for given $\theta \in (0, 1)$, 
$F_{K, N, D}(F_{K, N, D}^{-1}(\theta)/(1-\varepsilon))$ is strictly increasing in $\varepsilon$, and we have 
\[
\frac{d}{d\varepsilon} F_{K, N, D}\left(\frac{1}{1-\varepsilon} F_{K, N, D}^{-1}(\theta)\right) \Bigg{|}_{\varepsilon=0} = \D_{K, N, D}(\theta).
\]
\end{Prop}

\begin{proof}
The monotonicity of $F_{K, N, D}(F_{K, N, D}^{-1}(\theta)/(1-\varepsilon))$ in $\varepsilon$ immediately follows from the monotonicity of $F_{K, N, D}$. 
We also see that 
\begin{align*}
\frac{d}{d\varepsilon} F_{K, N, D}\left(\frac{1}{1-\varepsilon} F_{K, N, D}^{-1}(\theta)\right) \Bigg{|}_{\varepsilon=0}
 &= F^{-1}_{K, N, D}(\theta) F_{K, N, D}'\left( F_{K, N, D}^{-1}(\theta)\right) \\
 &= \frac{F^{-1}_{K, N, D}(\theta)}{(F_{K, N, D}^{-1})'(\theta)} = F^{-1}_{K, N, D}(\theta)I_{K, N, D}(\theta).
\end{align*}
Now, we set  
\[
H(\theta) := F_{K,N,D}^{-1}(\theta)I_{K, N, D}(\theta).
\] 
Then it is easy to observe that $H'(\theta) = \widetilde{\D}_{K, N, D} (\theta) H(\theta) +1$. 
Thus, it follows from the definition of $\widetilde{\D}_{K, N, D}$ that 
\begin{align}\label{eq: H}
\frac{H'(\theta)-1}{H(\theta)} = \widetilde{\D}_{K, N, D}(\theta) = \frac{(\D_{K, N, D})'(\theta) -1}{\D_{K, N, D}(\theta)}. 
\end{align}
Note that \eqref{eq: H} holds for any $\theta \in (0, 1)$. 
In order to prove our assertion, it suffices to prove that $H=\D_{K, N, D}$ holds on $[0, 1]$. 
We see that \eqref{eq: H} is equivalent to 
\begin{align}\label{eq: H2}
H'\D_{K, N, D}-H\D'_{K, N, D} = \D_{K, N, D}-H. 
\end{align}
Let $Y:=\{ \theta \in [0, 1]~|~ \D_{K, N, D}(\theta_1)=H(\theta_1) \text{ for any $\theta_1 \in [0, \theta]$}\}$. 
Note that $Y \neq \emptyset$ since $H(0)=\D_{K, N, D}(0)=0$. 
We also see that if $t_0 \in Y$, then we have $t_0 + \delta \in Y$ for small enough $\delta>0$ as follows. 
Fixed $t_0 \in Y$ with $t_0<1$, we suppose $t_0 + \delta \notin Y$ for any small enough $\delta>0$.  
Then there exists some $t_1 \in (t_0, 1]$ such that 
$\D_{K, N, D} > H$ or $\D_{K, N, D}<H$ holds on $(t_0, t_1)$. 
Without loss of generality, we may assume that $\D_{K, N, D} > H$ holds on $(t_0, t_1)$.
Then \eqref{eq: H2} implies that 
\[
(\log H)' > (\log \D_{K, N, D})'
\]
on $(t_0, t_1)$. 
Hence we obtain $H>\D_{K, N, D}$ on $(t_0, t_1)$ (when $t_0=0$, we use \eqref{limit}), which contradicts the assumption on $\D_{K, N, D}$ and $H$. 
Thus, $t_0 + \delta \in Y$ holds for small enough $\delta>0$, which implies that $Y$ is open in $[0, 1]$. 
Therefore, since $Y$ is closed by the continuity of $\D_{K, N, D}$ and $H$, we obtain $Y=[0, 1]$, which completes the proof.
\end{proof}

\begin{Lem}\label{lem: reduction}
Let $\mu$ be a probability measure supported on an open interval $I\subset \R$ with a continuous density on $I$. 
Then for any $ \varepsilon \in (0, 1)$, $\theta \in (0, 1)$ and interval $A \subset I$ with $\mu(A) = \theta$, there exists $\xi \in \overline{A}$ satisfying 
$\mu_-(A) = \mu_+(A) =\theta$ and 
\begin{align*}
\mu(A_{\varepsilon}) \geq \min \{ \mu_-((A\cap (-\infty, \xi])_{\varepsilon}), \mu_+((A\cap[\xi, \infty))_{\varepsilon}) \},
\end{align*}
where $\mu_-$ and $\mu_+$ are normalized probability measures of $\mu$ on $I \cap (-\infty, \xi]$ and $I \cap [\xi, \infty)$, respectively 
(when $\xi$ coincides with one of the endpoints of $I$, then we adopt $\mu(A_\varepsilon)$ as the right hand side above).
\end{Lem} 

\begin{proof}
Since the assertion is clear when $I \setminus A$ consists of one connected component, we may assume that $I \setminus A$ consists of two connected components. Moreover, without loss of generality, we may also assume that $A$ is closed. 
Let $G:I \to \R$ be the function defined as 
\[
G(x) := \mu((-\infty, x]\cap A)/\mu((-\infty, x])
\]
 and denote $A$ by $[a, b]$. 
Clearly, we have $G(a)=0$ and $G(b)=\theta/\mu((-\infty, b])> \theta$. Since $G$ is continuous, there exists some point $\xi \in \mathrm{int}(A)$ such that $G(\xi)= \theta$. 
Since it follows from the definition of the dilation that $A_{\varepsilon}$ includes the union of $(A\cap (-\infty, \xi])_{\varepsilon} \cap (-\infty, \xi]$ and $(A\cap [\xi, \infty))_{\varepsilon}\cap [\xi, \infty)$ whose intersection consists of only the element $\xi$, we see that 
\begin{align*}
\mu(A_{\varepsilon}) &\geq \frac{\mu((A\cap (-\infty, \xi])_{\varepsilon} \cap (-\infty, \xi])+ \mu((A\cap [\xi, \infty))_{\varepsilon}\cap [\xi, \infty))}{\mu((-\infty, \xi]) + \mu([\xi, \infty))} \\
	&\geq \min \left\{ \frac{\mu((A\cap (-\infty, \xi])_{\varepsilon} \cap (-\infty, \xi])}{\mu((-\infty, \xi])}, \frac{\mu((A\cap [\xi, \infty))_{\varepsilon}\cap [\xi, \infty))}{ \mu([\xi, \infty))} \right\} \\
		&= \min \{ \mu_-((A\cap (-\infty, \xi])_{\varepsilon}), \mu_+((A\cap[\xi, \infty))_{\varepsilon}) \},
\end{align*}
where we used the elementary inequality $(x_1 + x_2)/(x_3 + x_4) \geq \min \{ x_1/x_3, x_2/x_4 \}$ for any $x_i >0$ $(i=1, 2,3,4)$ in the second inequality.
On the other hand, since $G(\xi)=\theta$, we obtain $\mu_-(A)=\theta$ and, 
equivalently, 
\[
\mu_+(A) = \frac{\mu(A \cap [\xi, \infty))}{\mu([\xi, \infty))}
	= \frac{\theta - \mu(A \cap (-\infty, \xi])}{1-\mu((-\infty, \xi])} = \theta.
\]
This completes the proof. 
\end{proof}

Now, we shall prove Theorem \ref{thm: MT2}. 
It suffices to show the following theorem by the same argument as in Theorem \ref{MT0}.
The method of the proof  is derived from the isoperimetric inequality discussed by Bobkov and Houdr\'{e} in \cite[Theorem 2.1]{BH}.
\begin{Thm}\label{thm: MT2'}
Let $(I, |\cdot|, \mu)$ be an open interval $I\subset \R$ with a smooth density and satisfy $\Ric_{N}\geq K$ and $|I| \leq D$ for some $K \in \R$, $N \in (-\infty, 0]\cup (1, \infty]$ and $D \in (0, \infty]$. 
Assume that a triple $(K, N, D)$ satisfies Assumption (A). 
Then for any $\varepsilon \in (0, 1)$ and any interval $A \subset I$ with 
$\mu(A) < F_{K, N, D, \varepsilon}^{\infty}$, we have 
\[
\mu(A_{\varepsilon}) \geq F_{K, N, D}\left(\frac{1}{1-\varepsilon}F_{K, N, D}^{-1}(\mu(A))\right).
\]
\end{Thm}

\begin{proof}
Since the assertion is clear when $\mu(A)=0$, we may assume that $\mu(A) >0$.
For given $\theta \in (0, 1)$, 
we define $\tau(\theta) \in (0, 1)$ by the value 
$\sup\{\varepsilon\in(0, 1) ~|~ F_{K, N, D, \varepsilon}^{\infty}\geq \theta\}$
and $R_{\varepsilon}(\theta) := F_{K, N, D}(F_{K, N, D}^{-1}(\theta)/(1-\varepsilon))$ for any $\varepsilon \in (0, 1)$ and $\theta \in (0, F_{K, N, D, \varepsilon}^{\infty})$.
We remark that $F_{K, N, D, \varepsilon}^{\infty}$ is non-increasing in $\varepsilon$.
Now, fix $\theta \in (0, 1)$ and let $A$ be an interval in $I$ with $\mu(A) = \theta$. 
It suffices to prove that $\mu(A_{\varepsilon}) \geq R_{\varepsilon}(\mu(A))$ for any $\varepsilon \in (0, \tau(\theta))$. 
Instead of directly considering $R_{\varepsilon}$, we introduce 
\[
R_{\varepsilon}^{\sigma}(\vartheta) := F_{K, N, D}\left(\frac{1}{1-\varepsilon\sigma}F_{K, N, D}^{-1}(\vartheta)\right)
\]
for $\varepsilon, \sigma \in (0, 1)$ and $\vartheta \in (0, F_{K, N, D, \varepsilon}^{\infty})$, and we will show $\mu(A_{\varepsilon}) \geq R_{\varepsilon}^{\sigma}(\theta)$ for any $\varepsilon \in (0, \tau(\theta))$ and $\sigma \in (0, 1)$. 
Indeed, if this inequality holds, then letting $\sigma \to 1$ leads to $\mu(A_{\varepsilon}) \geq R_{\varepsilon}(\mu(A))$ for any $\varepsilon \in (0, \tau(\theta))$. 
In order to prove $\mu(A_{\varepsilon}) \geq R_{\varepsilon}^{\sigma}(\theta)$, by Lemma \ref{lem: reduction}, we may assume that $I$ is an interval $[0, \mu_{\infty})$ for some $\mu_{\infty} \in (0, \infty]$ and that $A$ is $[0, a]$ for some $a \in (0, \mu_{\infty})$ 
since the probability measures $\mu_-$ and $\mu_+$ constructed in Lemma \ref{lem: reduction} satisfy the same CDD condition that $\mu$ satisfies. 

Now, we fix $\sigma\in(0, 1)$. By the definition of the dilation area of $A$, we have 
\begin{align}\label{ineq: def}
\mu(A_{\varepsilon}) \geq \mu(A) +\mu^*(A)\varepsilon +o(\varepsilon). 
\end{align}
On the other hand, by the Taylor expansion of $R_{\varepsilon}^{\sigma}(\mu(A))$ at $\varepsilon=0$ and Proposition \ref{prop: concavity2}, we obtain 
\begin{align}\label{ineq: Taylor}
R_{\varepsilon}^{\sigma}(\mu(A)) = \mu(A) + \D_{K, N, D} (\mu(A))\varepsilon \sigma + o(\varepsilon).
\end{align}
Comparing \eqref{ineq: def} with \eqref{ineq: Taylor}, by Corollary \ref{MT1}, we see that there exists some small enough $\varepsilon_0 \in (0, 1)$ such that for any $\varepsilon \in (0, \varepsilon_0)$, 
$\mu(A_{\varepsilon}) \geq R_{\varepsilon}^{\sigma}(\mu(A))$ holds.

Let $\varepsilon_1 \in [0, \tau(\theta)]$ be the supremum of  the set of all $\bar{\varepsilon} \in (0, \tau(\theta))$ such that all $\varepsilon \in (0, \bar{\varepsilon})$ satisfy $\mu(A_{\varepsilon}) \geq R_{\varepsilon}^{\sigma}(\mu(A))$. By the argument above, we see that $\varepsilon_1>0$. 
Now, we suppose $\varepsilon_1<\tau(\theta)$, which will lead to a contradiction.
By the definition of $\varepsilon_1$, we have $\mu(A_{\varepsilon_1}) = R_{\varepsilon_1}^{\sigma}(\mu(A))$.
Since $\mu$ is supported on $I=[0, \mu_{\infty})$ and $A=[0, a]$, we see that $\mu(A_{\varepsilon_1}) = \mu([0, a/(1-\varepsilon_1)])$. 
We denote $[0, a/(1-\varepsilon_1)]$ by $B$. 
Then again by Corollary \ref{MT1} and the same argument above for $B$, we can take some small enough constant $\varepsilon_2 \in (0, 1)$ such that for any $\varepsilon \in (0, \varepsilon_2]$,  $\mu(B_{\varepsilon}) \geq R_{\varepsilon}^{\sigma}(\mu(B))$ holds. 
Now, fix a such $\varepsilon' \in (0, \varepsilon_2]$.
Since direct calculations imply  
\[
\mu(B_{\varepsilon'}) = \mu \left( \left[0, \frac{a}{(1-\varepsilon_1)(1-\varepsilon')}\right]\right) = \mu(A_{\varepsilon_3}), 
\]
where $\varepsilon_3 := 1- (1-\varepsilon_1)(1-\varepsilon') \in (0, 1)$, we obtain 
\begin{align}\label{ineq: R-mu}
R_{\varepsilon'}^{\sigma}(R_{\varepsilon_1}^{\sigma}(\mu(A))) = R_{\varepsilon'}^{\sigma}(\mu(B)) \leq \mu(B_{\varepsilon'}) = \mu(A_{\varepsilon_3}).
\end{align}
On the other hand, by the definition of $R_{\varepsilon}^{\sigma}$, we have 
\begin{align}\label{eq: R}
R_{\varepsilon'}^{\sigma}(R_{\varepsilon_1}^{\sigma}(\mu(A))) = R_{\varepsilon_3'}^{\sigma}(\mu(A)), 
\end{align}
where 
\[
\varepsilon_3' := \frac{1- (1-\varepsilon_1\sigma)(1-\varepsilon'\sigma)}{\sigma}.
\] 
Since $\sigma <1$ implies $\varepsilon_3' > \varepsilon_3$, we obtain $\mu(A_{\varepsilon_3}) \geq R_{\varepsilon_3}^{\sigma}(\mu(A))$ from \eqref{ineq: R-mu}, \eqref{eq: R} and the monotonicity of $R_{\varepsilon}^{\sigma}$ in $\varepsilon$. 
However,  this assertion contradicts the definition of $\varepsilon_1$ since $\varepsilon_3 = \varepsilon_1 +\varepsilon'(1-\varepsilon_1) > \varepsilon_1$ 
and $\varepsilon'$ is arbitrary in $(0, \varepsilon_2]$, 
where we may retake $\varepsilon_2$ such that $\varepsilon_1 +\varepsilon_2(1-\varepsilon_1) < \tau(\theta)$ holds if necessary.
Hence we obtain the desired assertion. 
\end{proof}

\section{Functional inequalities related to dilation profiles}

Some preceding investigations including Bobkov and Nazarov \cite{BN} and Fradelizi \cite{F} also studied the large and small deviation inequalities associated with certain parameters for a Borel function on $\R^n$ (more precisely, the modulus of regularity or the Remez function) via the $\varepsilon$-dilation inequalities, which are applied to establishing the Kahane-Khintchine type inequality.
In virtue of Thereom \ref{Klartag}, we can also see that the same inequalities hold under $CD(0, N)$ on a geodesically-convex $n$-dimensional weighted Riemannian manifold via the same arguments in the Euclidean setting.
In this section, we consider a new type of functional inequalities derived from the dilation profiles under $CD(0, N)$ with $N \in (-\infty, -1)\cup [n, \infty]$. 

\subsection{The case $N=\infty$}
Let $(\M, g, \m)$ be a geodesically-convex $n$-dimensional weighted Riemannian manifold. 
We first introduce the measured Remez function which is also used by Fradelizi \cite{F} without a measure.
\begin{Def}
Let $f :\M \to [0, \infty)$ be a Borel function. 
Given $s \geq 1$, we define $u_f(s)$ by the least constant $C\geq 1$ (including $\infty)$ satisfying   
\[
\m(\{ x \in \M ~|~ f(x) \leq \lambda\}_{1-1/s}) \leq \m(\{ x \in \M ~|~ f(x) \leq \lambda C \})
\]
for any $\lambda>0$. 
We say that a function $u_f:[1, \infty) \to [1, \infty]$ is the {\it measured Remez function} of $f$ if $u_f(s) < \infty$ for every $s\geq 1$ and it is continuous at $s=1$. 
\end{Def}

Equivalently, it holds that for any $\varepsilon \in (0, 1)$ and $\lambda>0$, 
\[
\m(\{ x \in \M ~|~ f(x) \leq \lambda\}_{\varepsilon}) \leq \m\left(\left\{ x \in \M ~\Bigg{|}~ f(x) \leq \lambda u_f\left(\frac{1}{1-\varepsilon}\right) \right\}\right).
\]
Every measured Remez function is non-decreasing and satisfies $u_f \geq 1$ on $[1, \infty)$ and $u_f(1)=1$. 
In addition, we define $u_f'(1)$ by 
\[
u_f'(1):=\limsup_{t\to1+0}\frac{u_f(t)-1}{t-1} >0.
\]
Note that a Borel function does not always have its measured Remez function. For instance, when $\m$ is the $n$-dimensional Lebesgue measure on $\R^n$, then the characteristic function $\mathbf{1}_A$ 
of any open proper subset $A \subset \R^n$ satisfies $u_{\mathbf{1}_A} = \infty$ on $(1, \infty)$. 
We can also deduce that for any $q, a>0$ and nonnegative Borel function $f$ with the measured Remez function, $u_{af^q}(s) = u_{f}(s)^q$ holds for every $s \in [1, \infty)$, which follows from the definition of the measured Remez function. Moreover since $u_{f}$ is continuous at $s=1$, we obtain $u_{af^q}'(1)= q u_{f}'(1)$.

\begin{Rem}\label{rem: rem3}
Fradelizi used the Remez function in \cite{F} depending only on a Borel function $f$. More precisely, 
for a given Borel function $f: \M \to \R$, its {\it Remez function} $\bar{u}_f:(1, \infty)\to [1, \infty)$ is defined as 
\[
\{ x \in \M ~|~ |f(x)| \leq \lambda\}_{1-1/s} \subset \{ x \in \M ~|~ |f(x)| \leq \lambda \bar{u}_f(s) \}
\]
for any $\lambda >0$. 
The definition of the Remez function immediately implies that $u_{|f|} \leq \bar{u}_f$ on $(1, \infty)$.
We also see that in general, these functions do not coincide. For instance, 
letting $(\R^2, \|\cdot\|_2^2, \m)$ be a weighted Riemannian manifold with a positive density on $\R^2$ and $f:\R^2 \to [0, \infty)$ be the characteristic function on $\R^2 \setminus I$ where $I$ is a closed segment in $\R^2$, we can deduce that $u_f \equiv 1$, but $\bar{u}_f \equiv \infty$ on $(1, \infty)$. 

\end{Rem}

According to \cite{F}, all norms $\|\cdot\|$ on $\R^n$ satisfy $u_{\|\cdot\|}(s)\leq2s-1$. 
More generally, all vector-valued polynomials $P$ of degree at most $d \geq 1$, namely
\begin{align}\label{eq: poly.}
P(x_1, x_2, \dots, x_n) = \sum_{i=1}^kP_i(x_1, x_2, \dots, x_n)e_i, \quad (x_1, x_2, \dots, x_n) \in \R^n, 
\end{align}
where $P_i(x_1, x_2, \dots, x_n)$ is a polynomial of degree at most $d$ for any $i=1, 2, \dots, k$ and $\{e_i\}_{i=1}^k$ is a basis in some normed vector space $(V, \|\cdot\|)$, satisfy $u_{\|P(\cdot)\|}(s) \leq T_d(2s-1)$, where $T_d$ is the Chebyshev polynomial of degree $d$ defined as 
\[
T_d(s) := \frac{ \left(s + \sqrt{s^2-1}\right)^d + \left(s - \sqrt{s^2-1}\right)^d}{2}
\]
(see also \cite{Bou}, \cite{B1}). We also note that the above estimates of the measured Remez functions are optimal for the Remez functions in the sense of Remark \ref{rem: rem3}. 
In particular, it is worth mentioning that under notations above, we have $u_{\|\cdot\|}'(1)\leq 2$ and $u_{\| P(\cdot)\|}'(1) \leq 2d^2$, which do not depend on the dimension of the base space. 

In order to describe our main claim in this subsection, we also introduce the relative entropy. 
Given a Borel function $f: \M \to [0, \infty)$, 
the relative entropy of $f$ with respect to $\m$ is defined by 
\begin{align*}
\Ent_{\m}(f) := \int_{\M}f\log f ~d\m - \int_{\M} f ~d\m\log\int_{\M} f~d\m.
\end{align*}
In particular, when $\nu = \rho \m$ is a probability measure on $\M$ where $\rho$ is a nonnegative Borel function on $\M$, 
then the relative entropy of $\nu$ with respect to $\m$ is defined by 
\[
\Ent_{\m}(\nu) :=\Ent_{\m}(\rho) =\int_{\M} \rho\log\rho ~d\m.
\]

Our main claim in this subsection is the following theorem.
\begin{Thm}\label{thm: entropy}
Let $(\M, g, \m)$ be a geodesically-convex $n$-dimensional weighted Riemannian manifold satisfying $\m(M)=1$ and $\Ric_{\infty} \geq 0$. 
Then for any probability measure $\nu:=\rho\m$ on $\M$, we have 
\[
\Ent_{\m}(\nu) \leq u_{\rho}'(1).
\]
\end{Thm}
Before proving this theorem, we describe a relation to the logarithmic Sobolev inequality which is one of well-known functional inequalities related to the relative entropy.
We say that a weighted Riemannian manifold $(\M, g, \m)$ satisfies the logarithmic Sobolev inequality with a constant $C>0$ if every probability measure $\nu=\rho \m$ on $\M$ whose density $\rho$ is locally Lipschitz  
satisfies 
\[
2C \Ent_\m(\nu) \leq I_\m(\nu), 
\]
where $I_\m(\nu)$ is the Fisher information defined by 
\[
I_\m(\nu) := \int_\M |\nabla \log \rho|^2 ~d\nu. 
\]
Equivalently, the above definition means that every locally Lipschitz function $f$ on $\M$ satisfies 
\[
\frac{C}{2} \Ent_\m(f^2) \leq \int_\M |\nabla f|^2 ~d\m. 
\]
In general, if $(\M, g, \m)$ satisfies $\Ric_\infty \geq K$ for some $K>0$, then it satisfies the logarithmic Sobolev inequality with the constant $K$. 
Under $\Ric_\infty \geq 0$, Theorem \ref{thm: entropy} yields the following logarithmic Sobolev type inequality. 
\begin{Cor}
Let $(\M, g, \m)$ be a geodesically-convex $n$-dimensional weighted Riemannian manifold satisfying $\m(M)=1$ and $\Ric_{\infty} \geq 0$.
We also assume that $(\M, g, \m)$ satisfies the Poincar\'{e} inequality with a constant $C>0$ in the sense that, for any locally Lipschitz function 
$h :\M \to \R$, it holds that 
\[
\int_\M h^2 ~d\m - \left(\int_\M h ~ d\m\right)^2\leq \frac{1}{C}\int_\M |\nabla h|^2 ~d\m. 
\]  
Let $f:\M\to\R$ be a locally Lipschitz function and set $a := \int_\M f ~d\m$. 
We assume that $|f-a|$ has the measured Remez function $u_{|f-a|}$. 
Then we have 
\[
\Ent_{\m}( f^2) \leq \frac{2u_{|f-a|}'(1) + 2}{C} \int_\M |\nabla f|^2 ~d\m. 
\]
\end{Cor}
\begin{proof}
Rothaus' lemma (for instance, see \cite[Lemma 5.1.4]{BGL}) yields that 
\[
\Ent_\m(f^2) \leq \Ent_\m((f-a)^2) + 2 \int_\M( f - a)^2 ~d\m. 
\]
Combining this inequality with Theorem \ref{thm: entropy} and $u_{|f-a|^2}'(1) = 2 u_{|f-a|}'(1)$, we obtain 
\[
\Ent_\m(f^2) \leq (2 u_{|f-a|}'(1) +2) \int_\M \left( f - \int_\M f~d\m\right)^2 ~d\m, 
\]
and finally, the Poincar\'{e} inequality yields 
\[
\Ent_{\m}( f^2) \leq \frac{2u_{|f-a|}'(1) + 2}{C} \int_\M |\nabla f|^2 ~d\m. 
\]
\end{proof}
Note that all log-concave probability measures on $\R^n$ (equivalently, $\Ric_\infty \geq 0$) satisfy the Poincar\'{e} inequality (for instance, see \cite{B99}). 
For the Poincar\'{e} inequality on a weighted Riemannian manifold with $\Ric_\infty\geq 0$, see \cite{Mi09}. 
In general, it is known that weighted Riemannian manifolds with $\Ric_\infty \geq 0$ do not always satisfy the logarithmic Sobolev inequality, and hence we need to add an appropriate assumption.  
For instance, the logarithmic Sobolev inequality under the Gaussian isoperimetric inequality is investigated in \cite{B99}. 
 
In order to show Theorem \ref{thm: entropy}, we first prove the following proposition which is regarded as a weak co-area type formula on dilation areas. 
For simplicity, given a Borel function $f:\M \to [0, \infty)$, we set $A_f(t) :=\{x \in \M ~|~ f(x) > t\}$ for $t\geq 0$.
\begin{Prop}\label{coarea type formula}
Let $(\M, g, \m)$ be a geodesically-convex $n$-dimensional weighted Riemannian manifold satisfying $\m(\M)=1$ and let $f :\M \to [0, \infty)$ be a Borel measurable function with the measured Remez function $u_f$. 
Then, we have 
\begin{align}\label{ineq: coarea}
\int_0^{\infty} \m^*(\M\setminus A_f(t)) ~dt \leq u_f'(1) \int_\M f ~d\m. 
\end{align}
\end{Prop}

\begin{proof}
We put $B(t):=\{ x \in \M ~|~ f(x) \leq t \}=\M \setminus A_f(t)$ for $t \geq 0$. 
By the definition of the measured Remez function, we deduce that 
\begin{align*}
\int_0^{\infty} \m^*(B(t)) ~dt &= \int_0^{\infty} \liminf_{\varepsilon \to 0} \frac{\m(B(t)_{\varepsilon}) -\m(B(t))}{\varepsilon} ~dt \leq \liminf_{\varepsilon \to 0} \int_0^{\infty}  \frac{\m(B(t)_{\varepsilon}) -\m(B(t))}{\varepsilon} ~dt \\
 & \leq \liminf_{\varepsilon \to 0} \int_0^{\infty}  \frac{\m\left(B\left(tu_f\left(\frac{1}{1-\varepsilon}\right) \right) \right) -\m(B(t))}{\varepsilon} ~dt \\
 &= \liminf_{\varepsilon \to 0} \frac{1}{\varepsilon}\int_0^{\infty} \left\{ \m( A_f(t) ) -\m\left(A_f\left(tu_f\left(\frac{1}{1-\varepsilon}\right)\right)\right)\right\} ~dt \\
 & =  \liminf_{\varepsilon \to 0} \frac{1}{\varepsilon} \left( 1-\frac{1}{u_f\left(\frac{1}{1-\varepsilon}\right)} \right) \int_{\M} f ~d\m \\
&\leq u_f'(1) \int_{\M} f ~d\m.
\end{align*}
\end{proof}

Note that Proposition \ref{coarea type formula} is optimal in the following sense. Let $\m$ be a probability measure on $\R_+:=[0, \infty)$
whose density is $e^{-x}$ and define $f: \R_+ \to [0, \infty)$ by $f(x)=x$. 
Since $[0, a]_{\varepsilon} \cap [0, \infty) = [0, a/(1-\varepsilon)]$ for any $a>0$ and $\varepsilon \in (0, 1)$, we can easily find that
$f$ has the measured Remez function of the form $u_f(s)=s$ for any $s \in [1, \infty)$, and hence we have $u_f'(1)=1$. 
We also see that $\int_{\R_+} f ~d\m$=1. 
Thus, the right hand side of \eqref{ineq: coarea} becomes $1$. 
On the other hand, since $\m$ and an interval $[0, \cdot]$ are the extremals of the dilation inequality \eqref{b} for $N=\infty$, and $A_f(t) =(t, \infty)$ for any $t\geq 0$, we obtain 
\[
\int_0^{\infty} \m^*(\M\setminus A_f(t)) ~dt =- \int_0^{\infty} \m(A_f(t)) \log\m(A_f(t)) ~dt = \int_0^\infty t e^{-t} ~dt =1, 
\]
where we used $\m(A_f(t))= \int_t^\infty e^{-x}~dx = e^{-t}$. 
Therefore, equality holds in \eqref{ineq: coarea} for $\m$ and $f$ above.

Now, we shall prove Theorem \ref{thm: entropy}. 
\begin{proof}[Proof of Theorem \ref{thm: entropy}]
Since $\Ric_{\infty} \geq 0$, it follows from \eqref{b} and Proposition \ref{coarea type formula} that we have 
\begin{align}\label{ineq: m1}
-\int_0^{\infty} \m(A_{\rho}(t)) \log\m(A_{\rho}(t)) ~dt \leq u_{\rho}'(1).
\end{align}
Now, recall the dual formula of the relative entropy (for instance, see \cite{V}): for any Borel function $f:\M \to [0, \infty)$ with $\int_\M f ~d\m =1$, 
it holds that 
\[
\Ent_\m(f) =\sup_{\varphi \in C_b(\M)}\left[ \int_\M f\varphi ~d\m - \log\int_\M e^{\varphi} ~d\m\right], 
\]
where $C_b(\M)$ is the set of all bounded continuous functions on $\M$. 
Hence, since $\Ent_{\m}(\m(A)^{-1}\mathbf{1}_A) = -\log\m(A)$ and $\int_\M \m(A)^{-1}\mathbf{1}_A ~d\m=1$ for any Borel subset $A \subset \M$ with $\m(A)>0$, the left hand side of \eqref{ineq: m1} becomes
\begin{align*}
-\int_0^{\infty} \m(A_{\rho}(t)) \log\m(A_{\rho}(t)) ~dt
 &= \int_0^\infty \m(A_\rho(t)) \Ent_\m(\m(A_\rho(t))^{-1}\mathbf{1}_{A_\rho(t)}) ~dt \\
 &= \int_0^\infty \sup_{\varphi \in C_b(\M)}\left[ \int_\M \varphi \mathbf{1}_{A_\rho(t)} ~d\m - \m(A_\rho(t))\log\int_\M e^{\varphi} ~d\m\right] ~dt\\
&\geq \sup_{\varphi \in C_b(\M)} \left[ \int_\M\int_0^\infty \varphi \mathbf{1}_{A_\rho(t)} ~dtd\m - \int_0^\infty \m(A_\rho(t))~dt \log\int_\M e^{\varphi} ~d\m\right] \\
&=\sup_{\varphi \in C_b(\M)} \left[ \int_\M \varphi\rho ~d\m - \log\int_\M e^{\varphi} ~d\m\right] \\
&=\Ent_\m(\nu), 
\end{align*}
where we used $\int_0^\infty \mathbf{1}_{A_\rho(t)}(x) ~dt=\rho(x)$ for every $x \in \M$ and $\int_0^\infty \m(A_\rho(t))~dt = \int_\M \rho~d\m =1$. 
\end{proof}
\begin{Rem}
We can replace the right hand side of \eqref{ineq: coarea} with a different form. 
Given a Borel function $f: \M \to [0, \infty)$ and $\varepsilon \in [0, 1)$, we define a function $f_{\varepsilon}: \M \to [0, \infty)$ by 
\[
f_{\varepsilon}(x) := \inf\{\lambda>0 ~|~ x \in f^{-1}([0, \lambda])_\varepsilon\}, \quad x \in \M.
\]
Note that $f_0 = f$ and $f_\varepsilon \leq f$ on $\M$ for any $\varepsilon \in (0, 1)$. 
We also define a function $\Phi_f : \M \to [0, \infty]$ by 
\[
\Phi_f(x) := \limsup_{\varepsilon \to 0} \frac{f(x) - f_\varepsilon(x)}{\varepsilon}, \quad x \in \M.
\]
For instance, when $f$ is a norm $\|\cdot\|_K$ on $\R^n$ whose unit ball is a centrally symmetric convex body $K \subset \R^n$, 
we can see that $f_\varepsilon = \|\cdot\|_{K_{\varepsilon}} = \|\cdot\|_{\frac{1+\varepsilon}{1-\varepsilon}K}$, and hence we obtain $\Phi_f=2\|\cdot\|_K$.

Now, given a Borel function $f: \M \to [0, \infty)$, when $f_\varepsilon$ is also a Borel function for small enough $\varepsilon >0$, then we can prove that 
\begin{align*}
\int_0^{\infty} \m^*(\M\setminus A_f(t)) ~dt \leq \int_\M \Phi_f ~d\m
\end{align*}
by the same argument as in Proposition \ref{coarea type formula} since we have $\{f \leq \lambda\}_\varepsilon \subset \{f_\varepsilon \leq \lambda\}$ for any $\lambda>0$ and $\varepsilon \in (0, 1)$. 
Moreover, combining this inequality with the argument in the proof of Theorem \ref{thm: entropy}, we obtain 
\begin{align*}
\Ent_\m(f) \leq \int_\M \Phi_f ~d\m.
\end{align*}
\end{Rem}

As the corollary of Theorem \ref{thm: entropy}, we describe the following Kahane-Khintchine type inequality.
\begin{Cor}\label{cor: reverse}
Let $(\M, g, \m)$ be a geodesically-convex $n$-dimensional weighted Riemannian manifold satisfying $\m(M)=1$ and $\Ric_{\infty} \geq 0$. 
Let $f:\M \to \R$ be an integrable function and the measured Remez function $u_{|f|}(s)$. 
Then for any $0<p\leq q<\infty$, we have  
\begin{align}\label{ineq: rev. Holder}
\left( \int_\M |f|^q ~d\m\right)^{1/q} \leq \left(\frac{q}{p}\right)^{u_{|f|}'(1)}\left(\int_\M|f|^p~d\m\right)^{1/p}.
\end{align}
\end{Cor}
\begin{proof}
We recall that given $q>0$ and $a>0$, we have $u_{a|f|^q}'(1)= q u_{|f|}'(1)$. 

Now, we define a function $\Lambda: (0, \infty) \to \R$ by 
\[
\Lambda(q):= \frac{1}{q}\log \left( \int_\M |f|^q ~d\m\right). 
\]
Then considering the probability measure $\mu_q := \rho_q \m$ with
\[
\rho_q(x) := \frac{|f(x)|^q}{\int_\M |f|^q ~d\m}, \quad x \in \M,
\]
we see that 
\begin{align*}
\Lambda'(q) = -\frac{1}{q^2}\log\left(\int_\M |f|^q ~d\m\right) + \frac{1}{q}\cdot\frac{\int_\M |f|^q\log|f| ~d\m}{\int_\M |f|^q ~d\m}
	= \frac{1}{q^2}\Ent_\m(\mu_q). 
\end{align*}
Thus it follows from Theorem \ref{thm: entropy} that we obtain 
\[
\Lambda'(q) \leq \frac{1}{q^2} u_{\rho_q}'(1) = \frac{1}{q} u_{|f|}'(1)
\]
for all $q>0$, which yields the desired assertion by integration.
\end{proof}
For $1\leq p\leq q<\infty$, H\"{o}lder's inequality yields $(\int_\M |f|^p ~d\m)^{1/p} \leq (\int_\M |f|^q ~d\m)^{1/q}$ for any Borel function $f$ on $\M$, and in this sense, 
\eqref{ineq: rev. Holder} is also mentioned as the reverse H\"{o}lder inequality. 
In general, it is well-known that Borell's lemma \eqref{Borell's lemma} yields the following reverse H\"{o}lder inequality on $\R^n$ (for instance, see \cite[Theorem 2.4.6]{BGVV}): 
\[
\left( \int_{\R^n} \|x\|^q ~d\mu(x)\right)^{1/q} \leq C\frac{q}{p}\left(\int_{\R^n}\|x\|^p~d\mu(x)\right)^{1/p}
\]
for any log-concave probability measure $\mu$ and norm $\|\cdot\|$ on $\R^n$, where $C>0$ is an absolute constant. 
On the other hand, under the same notations, \eqref{ineq: rev. Holder} yields that by $u_{\|\cdot\|}'(1)\leq2$, 
\begin{align}\label{ineq: distance}
\left( \int_{\R^n} \|x\|^q ~d\mu(x)\right)^{1/q} \leq \left(\frac{q}{p}\right)^2\left(\int_{\R^n}\|x\|^p~d\mu(x)\right)^{1/p}.
\end{align}
In particular, our inequality is meaningful when $p$ and $q$ are close to each other. 
Moreover, when $\mu_0$ is a probability measure on $[0, \infty)$ whose density with respect to the 1-dimensional Lebesgue measure is $e^{-x}$, since we see that for any $n \in \mathbb{N}$, the measured Remez function of the $\ell^\infty$-norm $\|\cdot\|_\infty$ with respect to $\mu_0^{\otimes n}$ in $[0, \infty)^n$ satisfies $u_{\|\cdot\|_\infty}(s)=s$ for every $s \geq 1$ by the same discussions after Proposition \ref{coarea type formula}, Corollary \ref{cor: reverse} yields 
\[
\left( \int_{\R^n} \|x\|_\infty^q ~d\mu_0^{\otimes n}(x)\right)^{1/q} \leq \frac{q}{p} \left(\int_{\R^n}\|x\|_\infty^p~d\mu_0^{\otimes n}(x)\right)^{1/p}
\]
for all $n \in \mathbb{N}$. 

More generally, the following reverse H\"{o}lder inequality for polynomials can be easily proved by combining Corollary \ref{cor: reverse} with the comments after Remark \ref{rem: rem3}. 
\begin{Cor}
Let $\Omega \subset \R^n$ be a convex open subset and $\mu$ be a log-concave probability measure supported on $\Omega$. 
We also take a normed vector space $(V, \|\cdot\|)$. Then for any vector-valued polynomial $P$ of degree at most $d\geq 1$ from $\Omega$ to $V$ defined as \eqref{eq: poly.} and $0<p\leq q<\infty$, 
we have 
\[
\left( \int_{\R^n} \|P(x)\|^q ~d\mu(x)\right)^{1/q} \leq \left(\frac{q}{p}\right)^{2d^2}\left(\int_{\R^n}\|P(x)\|^p~d\mu(x)\right)^{1/p}.
\]
\end{Cor}

We close this subsection by describing the reverse H\"{o}lder inequality for the distance function on a weighted Riemannian manifold corresponding to \eqref{ineq: distance}. 
Let $(\M, g)$ be a geodesically-convex $n$-dimensional  Riemannian manifold and $d_g$ be the distance function induced by $g$. 
Now, fix $x_0 \in \M$ and define $f:\M \to \R$ as $f(x) := d_g(x, x_0)$. 
Then we can deduce that $f$ has the Remez function in the sense of Remark \ref{rem: rem3} with $\bar{u}_{f}(s) \leq 2s-1$ for every $s\geq 1$ as follows. 
Denote by $B(r)$ the open ball centered at $x_0$ with a radius $r>0$. 
It suffices to prove that 
\begin{align}\label{ball}
B(r)_\varepsilon \subset B\left(\frac{1+\varepsilon}{1-\varepsilon}r\right) 
\end{align}
for any $r>0$ and $\varepsilon \in (0, 1)$.
First, note that given different two points $x, y \in \M$, letting $\gamma_{xy}: [0, 1] \to \M$ be a minimizing geodesic from $x$ to $y$, the triangle inequality yields $f(\gamma_{xy}(t)) \geq |(1-t)f(x)-tf(y)|$. 

Now, fix $r>0$ and $\varepsilon \in (0, 1)$, and take $x \in B(r)_\varepsilon$. 
By the definition of the $\varepsilon$-dilation, we can take $y \in \M$ such that $|B(r) \cap \gamma_{xy}| > 1-\varepsilon$ holds. 
In addition, we may assume that  $y$ belongs to $\overline{B(r)}$ by simple observations. 
Then we see that every $t \in [0, 1]$ with $\gamma_{xy}(t) \in B(r)$ satisfies 
\[
r \geq f(\gamma_{xy}(t)) \geq |(1-t)f(x) - tf(y)| \geq (1-t)f(x) - tf(y), 
\]
and hence we obtain 
\[
|B(r) \cap \gamma_{xy}| \leq \left| \left[\frac{f(x) -r}{f(x)+f(y)}, 1\right]\right| = \frac{r+f(y)}{f(x)+f(y)}. 
\]
Since we have $|B(r) \cap \gamma_{xy}| > 1-\varepsilon$, it yields 
\[
f(x) < \frac{r}{1-\varepsilon}+ \frac{\varepsilon}{1-\varepsilon}f(y) \leq \frac{1+\varepsilon}{1-\varepsilon}r, 
\]
which implies $x \in B((1+\varepsilon)r/(1-\varepsilon))$. Hence, we obtain \eqref{ball}. 
\begin{Rem}
In the Euclidean setting, it is known that equality holds in the left inclusion of \eqref{ball} for any $r>0$ and $\varepsilon \in (0, 1)$ (see \cite[Fact 1]{F}).  
However, we can easily observe that it does not always hold in general spaces. 
For instance, when $\M$ is the 1-dimensional unit sphere $\mathbb{S}^1$ with the canonical metric, 
we obtain that for any $r>0$ and $\varepsilon \in (0, 1)$, 
\begin{align*}
B(r)_\varepsilon = 
\begin{cases}
	B(\frac{1+\varepsilon}{1-\varepsilon}r) & \text{if $0<r<\frac{\pi}{2}(1-\varepsilon)$}, \\
	B(r+\varepsilon \pi) & \text{if $ r \geq \frac{\pi}{2}(1-\varepsilon)$}.
\end{cases}
\end{align*}
\end{Rem}
The following corollary is the reverse H\"{o}lder inequality for the distance function.
\begin{Cor}
Let $(\M, g, \m)$ be a geodesically-convex $n$-dimensional weighted Riemannian manifold satisfying $\m(M)=1$ and $\Ric_{\infty} \geq 0$ 
and fix $x_0 \in \M$. 
Then for any $0<p\leq q<\infty$, we have 
\[
\left( \int_\M d_g(x, x_0)^q ~d\m(x)\right)^{1/q} \leq \left(\frac{q}{p}\right)^2\left(\int_\M d_g(x, x_0)^p~d\m(x)\right)^{1/p}.
\]
\end{Cor}
\begin{proof}
By the above discussion and Remark \ref{rem: rem3}, we obtain $u_{d(\cdot, x_0)}(s) \leq 2s-1$ for every $s\geq 1$. 
Thus, our assertion follows from Corollary \ref{cor: reverse}.
\end{proof}
\subsection{The case $n \leq N <\infty$ and $-\infty<N<-1$}

In this final subsection, we discuss similar inequalities to Theorem \ref{thm: entropy} for more general $N \in (-\infty, -1)\cup[n, \infty)$.
For this purpose, we need to introduce an appropriate relative entropy. Although in general, it is natural to consider the R\'{e}nyi entropy for $N \in (-\infty, -1) \cup[n, \infty)$ in the context of geometric analysis, here we use other entropy (for instance, see Simon \cite[Chapter 16]{S}). 
Let $(\M, g, \m)$ be a geodesically-convex $n$-dimensional weighted Riemannian manifold with $\m(\M)=1$. 
For a probability measure $\nu=\rho \m$ on $\M$, where $\rho$ is a nonnegative Borel function on $\M$ with $\rho^{(1+N)/N} \in L^1(\m)$, we define $U_N(\nu)$ by
\[
U_N(\nu) =U_N(\rho):= N\int_{\M} (\rho^{1/N}-1)\rho ~d\m = N\int_{\M} \rho^{(1+N)/N} ~d\m -N
\]
for every $N \in (-\infty, -1)\cup [n, \infty)$. In this paper, we call the above entropy the {\it $N$-entropy}. Note that the function $(0, \infty) \ni x \mapsto N(x^{1/N}-1)x$ is convex, and hence Jensen's inequality yields
that $U_N(\nu) \geq 0$. 
Moreover, for $-\infty<N'<-1$ and $n \leq N< \infty$, the $N$-entropy enjoys that 
\[
0\leq U_{N'}(\nu) \leq \Ent_{\m}(\nu) \leq U_N(\nu), 
\]
$U_N(\nu)\to\Ent_{\m}(\nu)$ as $N \to \infty$ (similarly, $U_{N'}(\nu)\to\Ent_{\m}(\nu)$ as $N' \to -\infty$) and $U_{N'}(\nu)\to 0$ as $N' \to -1$. 
The following theorem corresponds to Theorem \ref{thm: entropy}.  
\begin{Thm}\label{thm: general entropy}
Let $(\M, g, \m)$ be a geodesically-convex $n$-dimensional weighted Riemannian manifold with $\m(\M)=1$, $\Ric_N\geq 0$ for some $N \in (-\infty, -1)\cup [n, \infty)$. 
Then for any probability measure $\nu$ on $\M$ whose density with respect to $\m$ is $\rho$, we have
\[
U_N(\nu) \leq u_{\rho}'(1).
\]
\end{Thm}
In order to prove Theorem \ref{thm: general entropy}, we introduce the dual formula of the $N$-entropy that we postpone proving to the end of this subsection. 
\begin{Thm}\label{Simon}
Let $(\M, g, \m)$ be a geodesically-convex $n$-dimensional weighted Riemannian manifold with $\m(\M)=1$. 
Then for any probability measure $\nu$ on $\M$ and $N \in (-\infty, -1) \cup [n, \infty)$, it holds that 
\begin{align}\label{thm: dual formula}
U_N(\nu) = \sup_{g \in \E_N(\M)} \left[(1+N)\int_\M g ~d\nu - \int_\M g^{1+N} ~d\m\right] -N,
\end{align}
where $\E_N(\M)$ is the set of all nonnegative measurable functions $g$ with $g^{1+N} \in L^1(\m)$ when $N \in [n, \infty)$, and the set of all continuous functions $g$ with $\inf_{x \in \M} g(x)>0$ when $N \in (-\infty, -1)$. 
\end{Thm}
%

Now, we shall prove Theorem \ref{thm: general entropy}. 
\begin{proof}[Proof of Theorem \ref{thm: general entropy}]
The idea of the proof is same as Theorem \ref{thm: entropy}. By \eqref{b} and Proposition \ref{coarea type formula}, 
it suffices to prove 
\[
-N \int_0^\infty (\m(A_\rho(t)) - \m(A_\rho(t))^{1-1/N}) ~dt \geq U_N(\nu). 
\]
In general for a Borel subset $A \subset \M$ with $\m(A)>0$, we have 
\[
U_N\left(\frac{1}{\m(A)}\mathbf{1}_A\right) = N\m(A)^{-1/N} -N. 
\]
Thus, we see that by Theorem \ref{Simon}, 
\begin{align*}
&-N \int_0^\infty (\m(A_\rho(t)) - \m(A_\rho(t))^{1-1/N}) ~dt \\
&= \int_0^\infty \m(A_\rho(t))U_N\left(\frac{1}{\m(A_\rho(t))}\mathbf{1}_{A_\rho(t)}\right) ~dt\\
&= \int_0^\infty \sup_{g \in \E_N(\M)}\left[ (1+N)\int_{\M} g \mathbf{1}_{A_\rho(t)} ~d\m - \m(A_\rho(t))\int_{\M} g^{1+N}~d\m -N \m(A_\rho(t)) \right] ~dt \\
&\geq \sup_{g \in \E_N(\M)}\left[ (1+N)\int_{\M}\int_0^\infty g \mathbf{1}_{A_\rho(t)} ~dtd\m - \int_0^\infty\m(A_\rho(t)) ~dt \int_{\M} g^{1+N}~d\m -N\int_0^\infty\m(A_\rho(t)) ~dt \right] \\
&=  \sup_{g \in \E_N(\M)}\left[ (1+N)\int_{\M} g ~d\nu - \int_{\M} g^{1+N}~d\m -N \right]\\
&=U_N(\nu), 
\end{align*}
where we used $\int_0^\infty \mathbf{1}_{A_\rho(t)}(x) ~dt=\rho(x)$ for every $x \in \M$ and $\int_0^\infty \m(A_\rho(t))~dt = \int_\M \rho~d\m =1$. 
This completes the proof. 
\end{proof} 
\begin{proof}[Proof of Theorem \ref{Simon}]
Let $\rho$ be the density of $\nu$ with respect to $\m$ satisfying $\rho^{(1+N)/N} \in L^1(\m)$. 
First, let $N \in [n, \infty)$. 
The Young inequality implies that  
\[
xy \leq \frac{N}{1+N}x^{(1+N)/N} + \frac{1}{1+N}y^{1+N}, \quad x\geq 0, \quad y\geq 0.
\]
Thus, for any measurable function $g:\M \to \R_+$ with $g^{1+N} \in L^1(\m)$, we have   
\[
\rho g \leq \frac{N}{1+N}\rho^{(1+N)/N} + \frac{1}{1+N}g^{1+N},
\]
which yields 
\[
\frac{N}{1+N}\int_\M \rho^{(1+N)/N} ~d\m \geq \sup_{g \in \E_N(\M)} \left[\int_\M g ~d\nu - \frac{1}{1+N}\int_\M g^{1+N} ~d\m\right]. 
\]
On the other hand, letting $g=\rho^{1/N}$ (which implies $g^{1+N} = \rho^{(1+N)/N} \in L^1(\m)$) yields equality in the above inequality. 
Consequently, we obtain 
\[
\frac{N}{1+N}\int_\M \rho^{(1+N)/N} ~d\m = \sup_{g \in \E_N(\M)} \left[\int_\M g ~d\nu - \frac{1}{1+N}\int_\M g^{1+N} ~d\m\right], 
\]
and hence 
\[
U_N(\nu) = \sup_{g \in \E_N(\M)} \left[(1+N)\int_\M g ~d\nu - \int_\M g^{1+N} ~d\m\right] -N.
\]
\\
\indent
Second, let $N \in (-\infty, -1)$. 
Note that 
\[
\varphi_N(x) :=
\begin{cases}
 -\frac{N}{1+N}x^{(1+N)/N} & \text{if $x \geq0$}, \\
 \infty & \text{if $x <0$}
\end{cases}
\]
is convex and lower semi-continuous on $\R$. 
Its Legendre transform $\varphi_N^*(y) :=\sup_{x \in \R} [xy- \varphi_N(x)]$ has the explicit form
\[
\varphi_N^*(y) =
\begin{cases}
-\frac{1}{1+N} (-y)^{1+N} & \text{if $y< 0$}, \\
\infty & \text{if $ y\geq0$}.
\end{cases}
\]
Therefore, we obtain the reverse Young type inequality 
\[
xy \geq \frac{N}{1+N}x^{(1+N)/N} + \frac{1}{1+N}y^{1+N}, \quad x\geq 0, \quad y>0.
\]
Thus, for any continuous function $g:\M \to \R$ with $\inf_{x \in \M} g(x)>0$, we have 
\[
\rho g \geq \frac{N}{1+N}\rho^{(1+N)/N} + \frac{1}{1+N}g^{1+N},
\]
which yields 
\[
\frac{N}{1+N}\int_\M \rho^{(1+N)/N} ~d\m \leq \inf_{g \in \E_N(\M)} \left[\int_\M g ~d\nu - \frac{1}{1+N}\int_\M g^{1+N} ~d\m\right]. 
\]
On the other hand, if $\rho$ is continuous and satisfies $\inf_{x \in \M}\rho^{1/N}(x)>0$ (namely, $\sup_{x \in \M}\rho(x)<M$ for some $M>0$), then letting $g=\rho^{1/N}$ yields equality in the above inequality.

In general, we use an approximation argument. 
For $g: \M \to \R_+$, we set 
\[
S(g) := \int_\M g ~d\nu - \frac{1}{1+N}\int_\M g^{1+N} ~d\m. 
\]
Let $\overline{\E}(\M)$ be the set of all Borel measurable functions $h$ satisfying $\delta\leq h\leq\delta^{-1}$ for some $\delta \in (0, 1)$. 
For such $h$, we can take a sequence of continuous functions $\{h_k\}_{k=1}^\infty$ such that $h_k \to h$ as $k\to \infty$ $d\m$-a.e. in the pointwise sense. 
By replacing $h_k$ by $\min\{\delta^{-1}, \max\{h_k, \delta\}\}$, we may assume that $h_k$ satisfies $\delta\leq h_k\leq\delta^{-1}$. 
In particular, $h_k$ belongs to $\E_N(\M)$ for all $k$. By the dominated convergence theorem, we have $S(h_k) \to S(h)$ as $k \to \infty$. 
Thus, we obtain 
\[
\inf_{g \in \E_N(\M)} \left[\int_\M g ~d\nu - \frac{1}{1+N}\int_\M g^{1+N} ~d\m\right] \leq \inf_{g \in \overline{\E}(\M)} \left[\int_\M g ~d\nu - \frac{1}{1+N}\int_\M g^{1+N} ~d\m\right].
\]
Therefore, for proving \eqref{thm: dual formula}, it suffices to find a sequence $\{g_k\}_{k=1}^\infty$ in $\overline{\E}(\M)$ such that 
\begin{align}\label{converge}
S(g_k) \rightarrow \frac{N}{1+N}\int_\M \rho^{(1+N)/N} ~d\m \quad \text{as $k \rightarrow \infty$}.
\end{align}
Now, we define $g_k: \M \to \R$ for every $k \in \mathbb{N}$ by 
\[
g_k(x):=
\begin{cases}
k^{-1} & \text{if $\rho^{1/N}(x) \leq k^{-1}$}, \\
\rho^{1/N}(x) & \text{if $k^{-1}\leq \rho^{1/N}(x) \leq k$}, \\
k & \text{if $\rho^{1/N}(x) \geq k$}.
\end{cases}
\]
Obviously, for fixed $x \in \M$, $g_k(x)$ is non-decreasing in $k$ if $\rho^{1/N}(x)\geq 1$ and non-increasing if $\rho^{1/N}(x) \leq 1$. 
We can also see that for fixed $y>0$, the function $\R_+ \ni z \mapsto zy-z^{1+N}/(1+N) \in \R$ is increasing on $[y^{1/N}, 1]$ if $y^{1/N}<1$ and decreasing on $[1, y^{1/N}]$ if $y^{1/N}>1$. 
Therefore, it follows that for every $x \in \M$, $g_k(x)\rho(x) - g_k^{1+N}(x)/(1+N)$ is non-increasing in $k$ and converges to $N\rho^{(1+N)/N}/(1+N)$. 
Hence, by the monotone convergence theorem, we obtain \eqref{converge} for $\{g_k\}_{k=1}^\infty$ defined above.
Consequently, we have proved 
\[
\frac{N}{1+N}\int_\M \rho^{(1+N)/N} ~d\m = \inf_{g \in \E_N(\M)} \left[\int_\M g ~d\nu - \frac{1}{1+N}\int_\M g^{1+N} ~d\m\right], 
\]
and since $1+N <0$, it yields 
\[
U_N(\nu) = \sup_{g \in \E_N(\M)} \left[(1+N)\int_\M g ~d\nu - \int_\M g^{1+N} ~d\m\right] -N.
\]
\end{proof}

\end{document}